\documentclass[10pt,a4paper]{article}
\usepackage[latin1]{inputenc}
\usepackage{amssymb,amsthm,amsmath,amsfonts,mathrsfs}
\usepackage{hyperref}
\usepackage{makeidx}
\usepackage{graphicx}
\usepackage{fancyvrb}
\usepackage{color}
\usepackage[width=16.00cm, height=24.00cm]{geometry}

\newtheorem{prop}{Proposition}[section]
\newtheorem{lem}{Lemma}[section]
\newtheorem{coro}{Corollary}[section]
\newtheorem{thm}{Theorem}[section]
\numberwithin{equation}{section}

\newcommand{\eps}{\varepsilon}

\newtheorem{rmq}{\textbf{Remark}}[section]

\usepackage{fancyhdr}
\newcommand{\E}{\mathcal{E}}

\newcommand{\be}{\begin{equation} \label}
\newcommand{\ee}{\end{equation}}
\newcommand{\R}{\mathbb{R}}

\author{Loth}
\begin{document}
	
	\begin{center}		
	\section*{Refined blow-up behavior for reaction-diffusion equations with non scale invariant exponential nonlinearities} 
		$ $
	\end{center}
	
	\begin{center}
	Loth Damagui CHABI 
		\footnote{Universit\'e Sorbonne Paris Nord \& CNRS UMR 7539, Laboratoire Analyse G\'eom\'etrie et Applications, 93430 Villetaneuse,  France.
	Email: chabi@math.univ-paris13.fr}\\
		$ $
	\end{center}
	
 \begin{abstract}  
 	We consider positive radial decreasing blow-up solutions of the semilinear heat equation \begin{equation*}
 u_t-\Delta u=f(u):=e^{u}L(e^{u}),\quad x\in \Omega,\ t>0,
 \end{equation*}
 where  $\Omega=\R^n$ or $\Omega=B_R$ and $L$ is a slowly varying function
 (which includes for instance logarithms and their powers and iterates, as well as some strongly oscillating  unbounded functions). 
 We characterize the asymptotic blow-up behavior and obtain the sharp, global blow-up profile in the scale of the original variables $(x, t)$. 
Namely, 
 assuming for instance $u_t\ge 0$, we have \begin{equation*}
 	u(x,t)=G^{-1}\bigg(T-t+\frac{1}{8}\frac{|x|^2}{|\log |x||}\bigg)+o(1)\ \ \hbox{as $(x,t)\to (0,T)$,\quad where } 
	 G(X)=\int_{X}^{\infty} \frac{ds}{f(s)}.	
 	\end{equation*}
	 This estimate in particular provides the sharp final space profile and the refined space-time profile. 
For exponentially growing nonlinearities, such results were up to now available only in the scale invariant case $f(u)=e^u$.
Moreover, this displays a universal structure of the global blow-up profile, given by the resolvent $G^{-1}$ of the ODE composed with a
fixed time-space building block, which is robust with respect to the factor $L(e^u)$.
\end{abstract}
\vskip 6pt

\noindent{\bf Key words:} Semilinear heat equation,  exponential nonlinearities, asymptotic blowup behavior, blow-up profile, refined space-time behavior, slow variation.

\tableofcontents

\section{Introduction}

This paper is a contribution to the understanding of the precise singularity formation 
	for solutions of superlinear parabolic problems involving {\it non scale invariant nonlinearities},
	in continuation to the works \cite{duong2018construction,hamza2022blow,souplet2022universal, chso, QSLiouv, chabi1} (see below for more details on these works).
	
	\smallskip
	
Consider the following semilinear heat equation 
	\begin{equation}
	\begin{cases}
	u_t-\Delta u=f(u),\quad x\in \Omega,\ t>0,\\
	u=0,\quad x\in \partial\Omega,\ t>0,\\
	u(x,0)=u_0(x),\quad x \in \Omega,
	\end{cases},\label{eqE1}
	\end{equation}
	where $\Omega \subset \mathbb{R}^n\ (n\ge 1)$ is smooth, $u_0\in L^\infty(\Omega)$  and $f\in C^1([0,\infty))$.
For the model cases 
$$\hbox{$f(u)=|u|^{p-1}u$ \quad ($p>1$) \quad and \quad $f(u)=e^u$},$$
the blowup behavior has been the subject of extensive research
	and is rather well understood, especially in the Sobolev subcritical power case $1<p<(n+2)/(n-2)_+$
	 where a complete classification of blow-up profiles is available.
	For $f(u)=|u|^{p-1}u$ we refer to \cite{bebernes1987description,BK94,filippas1992refined,giga1985asymptotically,giga1987characterizing,
giga1989nondegeneracy,MM2004a,MM2004b,merle1998optimal,merle1998refined,PQS07b,Q21,quittner2019superlinear,souplet2019simplified,Vel93b}
	 and, for $f(u)=e^u$, to \cite{bebernes1987description,bebernes1992final, bressan1990asymptotic,BT,E,fila2008nonconstant,
friedman1985blow,HV93,pulkkinen2011blow,souplet2022refined,W}.
	However the analysis of sharp blow-up asymptotics in these works heavily depends on the exact form of the nonlinearities and their scale invariance properties, 
	namely the invariance of the equations under the transformations 
	$$u_\lambda(x,t)=\lambda^{2/(p-1)}u(\lambda x, \lambda^2 t) \quad (\lambda >0),$$
and
$$u_\lambda(x,t)=u(\lambda x, \lambda^2 t)+\log \lambda \quad (\lambda >0)$$
respectively (or at least on an asymptotic homogeneity of the nonlinearities as the variable $u\to\infty$).

	\smallskip
	
 Let us give a brief review of known results on the blowup behavior of solutions of problem \eqref{eqE1} 
for non scale-invariant nonlinearities.
In the early work \cite{friedman1985blow},
blow-up rate estimates were obtained for time increasing positive solutions of \eqref{eqE1} with rather general nonlinearities.
For suitable classes of non scale-invariant nonlinearities,
results on blow-up sets were obtained in \cite{Fuji}.
In \cite{duong2018construction,hamza2022blow},
for problem \eqref{eqE1} with a particular, logarithmically perturbed nonlinearity of the form $f(u)=|u|^{p-1}u\log(2+u^2)$,
 blowup rate estimates were obtained (without sign or monotonicity assumption),
 and special solutions with a prescribed blowup profile were constructed
 (see also \cite{hamza2021, RZ} for blowup rate estimates for the related nonlinear wave equations).
In \cite{chabi1,chso,QSLiouv,souplet2022universal}, positive solutions were studied for problem \eqref{eqE1} with general regularly varying nonlinearities, 
of the form $f(u)=u^pL(u)$ with $L$ of slow variation (see \eqref{prop} for precise definition).
In \cite{QSLiouv,souplet2022universal}, 
parabolic Liouville type theorems were obtained, along with applications to universal bounds and blowup rate estimates.
An analogue of the classical Giga-Kohn result was obtained in \cite{chabi1},
using i.a.~the type~I blowup estimate from \cite{souplet2022universal}.
Moreover, building on the result in \cite{chabi1}, the final and space-time blow-up profile of positive radially decreasing solutions
 was obtained in \cite{chso}, revealing
a {\it structural universality} of the blow-up profile,  robust with respect to the factor $L(u)$.

	\smallskip

In this paper, we are interested in non-scale invariant nonlinearities with exponential growth at infinity. 
 We note that, concerning blowup rate estimates, the results in \cite{friedman1985blow} and in the recent work \cite{FK}
cover such nonlinearities (under suitable assumptions).
Our main goal here is to provide a very precise description of the blowup asymptotics, 
including space-time and final blowup profiles,
   for radially decreasing solutions of problem \eqref{eqE1},
  for a large class of exponentially growing nonlinearities without scale invariance.
   Our analysis follows in part that in our previous works \cite{chabi1,chso} for power type regularly varying nonlinearities.

	 \section{Main results}

Assume 
	 \begin{equation}\label{sem0}
	 f\in C^1([0,\infty)),\quad f(0)\ge 0, \quad\hbox{$f$ is $C^2$ and positive for large $s$.} 
	 \end{equation}
	By standard theory, 
	 for $u_0\in L^\infty(\Omega)$ with $u_0\ge 0$, problem \eqref{eqE1} has a unique, nonnegative classical solution, of maximal existence time $T = T(u_0) \in (0,\infty],$ that we will denote by $u$ throughout this paper. 
	 If $1/f$ is integrable at infinity and $u_0$ is suitably large, then $u$ blows up in finite time, i.e $ T < \infty$ and
	\begin{equation*}
	\lim_{t\to T} \|u(t)\|_\infty=\infty.
	\end{equation*}
	 In this case, $T$ is called the blowup time of $u$. Given $a \in \overline{\Omega}\subset\mathbb{R}^n$, we say that $a$ is a blowup point of $u$ if there exists $(a_j,t_j)\to(a,T)$ such that $u(a_j,t_j)\to\infty$ as $j\to \infty$. 

 We consider nonlinearities of the form 
\begin{equation}\label{2.2}
f(s)=e^sL(e^s), \quad s\ge 0.
\end{equation}
 The function $L$ will be assumed to satisfy
	 \begin{equation}\label{sem}	
\frac{ sL'(s)}{L(s)}=o(\log^{-\alpha}(s))\quad\hbox{and}\quad \Bigl(\frac{ sL'(s)}{L(s)}\Bigl)'=o\Bigl(\frac{1}{s\log s}\Bigr)
	 \quad\hbox{ as $s\to\infty$, \quad for some $\alpha>\frac{1}{2}$.}
	 \end{equation}
	   Such $L$ 
	   include logarithms and their powers and iterates, as well as some strongly oscillating functions
	    (see Remark~\ref{ex1} below for typical examples).  Note that \eqref{sem} is a subclass of the functions with {\it slow variation}, i.e 
\begin{equation}\label{prop}
	\lim_{\lambda\to\infty}\frac{L(\lambda s)}{L(\lambda)}=1\quad \text{for each }\ s>0
	\end{equation}
and that, for  $L\in C^1$ near infinity, a sufficient condition for \eqref{prop} is $\lim_{s\to\infty} s\frac{L'(s)}{L(s)}=0$.

 To state  our results we introduce the following function \begin{equation}\label{f}
	 G(X):=\int_{X}^{\infty}\frac{ds}{f(s)}.  
	 \end{equation} 
	 Note that, under assumptions \eqref{sem0}-\eqref{sem}, $G$ is well defined and there exists a large $A>0$ such that $G:[A,\infty)\to (0,G(A)]$ is decreasing. Moreover, 
 \be{defODE-sol}	 
 \psi(t):=G^{-1}(T-t)
\ee
 is the positive solution of the corresponding ODE, i.e $y'(t)=f(y)$, which blows up  at $t=T$.  
	 	Concerning $u$ we will focus on positive solutions that are radial decreasing in space and increasing in time. 
		 To this end, we will assume
			\begin{equation}\label{Omega-u0}
		\hbox{$\Omega=\mathbb{R}^n$ or $\Omega=B_R$, $u_0\in L^\infty(\Omega)$, $u_0\ge 0$,}
					 \end{equation} 
					\begin{equation}\label{Omega-u0b}
		\hbox{$u_0$ radially symmetric, nonincreasing in $r = |x|$, with $u_0$ nonconstant if $\Omega=\R^n$.}
			 \end{equation} 
This guarantees that $u(x,t)=u(r,t)\ge 0$ with $r=|x|$ and $u_r\le 0$.
As for the property $u_t\ge0$, we recall that it is satisfied for instance when $u_0\in BC(\bar{\Omega})\cap H^2_{loc}(\Omega)$ is such that $\Delta u_0+f(u_0)\ge 0$ and $u_0=0$ on $\partial \Omega$ (see, e.g.,~\cite[Proposition 52.19]{quittner2019superlinear}).
	 
	 \smallskip
	 
	  Our first main result is the following global, refined blowup estimate, valid in the scale of the original variables~$(x, t)$.

\begin{thm}\label{Thm}
 Assume \eqref{sem0}-\eqref{sem}, \eqref{Omega-u0}, \eqref{Omega-u0b}, $T=T(u_0)<\infty$  
	and $u_t\ge 0$.
	Then 
	\begin{equation}\label{infer1}
	u(x,t)=G^{-1}\bigg(T-t+\frac{|x|^2}{4|\log|x|^2|}\bigg)+o(1)\quad \text{as } (x,t)\to (0,T).
	\end{equation}
\end{thm}

 Combining \eqref{infer1} with the asymptotic formula of $G^{-1}$ (see Lemma \ref{lm1}(ii) below), 
 we obtain the explicit blow-up asymptotics of $u$, which takes the form
$$u(x,t)= -\log\bigg(T-t+\frac{|x|^2}{4|\log|x|^2|}\bigg)-\log\Bigg\{ L\bigg(\frac{1}{T-t+\frac{|x|^2}{4|\log|x|^2|}}\bigg)\Bigg\}+o(1)\quad \text{as } (x,t)\to (0,T).$$
The latter displays the precise influence of the factor $L$ on the profile.
For instance for the typical case 
$$f(u)=u^q e^u,\quad q\ge 1,$$
corresponding to $L(s)=(\log s)^q$, we obtain
$$u(x,t)= -\log\bigg(T-t+\frac{|x|^2}{4|\log|x|^2|}\bigg)-q\log\Bigg\{-\log\bigg(T-t+\frac{|x|^2}{4|\log|x|^2|}\bigg)\Bigg\}+o(1)\quad \text{as } (x,t)\to (0,T).$$
 See Remark~\ref{ex1} below for more examples.

On the other hand,  we note from \eqref{infer1} that $u$ blows up only at $x=0$ hence, by standard parabolic estimates, the limit $u(x, T):= \lim_{t\to T} u(x, t)$ exists
for all $x\neq 0$. As direct consequence of Theorem \ref{Thm}, we obtain the final blow-up profile and the corresponding refined space-time profile of the solution. 

\begin{coro}\label{coro}
 	Under the assumptions of Theorem \ref{Thm}, the following holds:
	\begin{enumerate}
		\item [(i)] {\bf(Final profile estimate)}
		\begin{equation}\label{fp}
		u(x,T)=G^{-1}\Bigg(\frac{|x|^2}{ 4|\log|x|^2|}\Bigg)\quad\text{as } x\to 0. 
		\end{equation}
		\item [(ii)]{\bf(Refined space-time profile estimate)}
		\begin{equation}\label{fp1}
		u(\xi\sqrt{(T-t)|\log(T-t)|},t)= G^{-1}\bigg[(T-t)\bigg(1+\frac{|\xi|^2}{4}\bigg)\bigg]+o(1) \quad\text{as }\ t\to T
		\end{equation}
		 uniformly for $\xi$ bounded.
	\end{enumerate}
\end{coro}

In the special case of the pure exponential, we have $G^{-1}(Y)=-\log Y$ for $Y>0$ and Corollary \ref{coro} recovers the known final and refined profiles (cf.~\cite{bebernes1992final,fila2008nonconstant,HV93,Vel93b}).
As for the global profile in Theorem \ref{Thm}, it seems new even for the pure exponential 
and completes the upper estimate that was obtained in \cite{souplet2022refined}.

In the course of the proof of Theorem \ref{Thm},
we establish the following expansion of $u(y\sqrt{T-t},t)-\psi(t)$ with second key order term for bounded $y$. 
This is of independent interest 
since it gives more precise expansion of the solution than \eqref{infer1}, but only in the more restricted range 
$|x|=O(\sqrt{T-t})$. 
In the special case $f(u)=e^u$ this result was proved in \cite{bebernes1992final,fila2008nonconstant,HV93}.
 Here and in the rest of this paper, we denote 
	\begin{equation}\label{defL2H1}
	L^2_\rho:=\Bigl\{v\in L^2_{loc}(\mathbb{R}^n);\ \int_{\mathbb{R}^n}v^2(y)\rho(y)dy<\infty\Bigr\},\quad 
	H^1_\rho:=\bigl\{f\in L^2_\rho;\ \nabla f\in L^2_\rho\bigr\},\quad \rho(y):=e^{-|y|^2/4},
	\end{equation}
	and $u$ is defined to be $0$ outside $\Omega$ in case $\Omega=B_R$.

\begin{thm}\label{theo}
	Under assumptions of Theorem \ref{Thm}, we have
	\begin{equation}\label{A1}
	u(y\sqrt{T-t},t)-\psi(t)= \frac{ 2n-|y|^2+o(1)}{4|\log(T-t)|},\quad \text{as} \ t\to T,
	\end{equation}
	with convergence in $H^1_\rho(\mathbb{R}^n)$ and uniformly for bounded $y$, where $\psi$ is as in \eqref{defODE-sol}.
\end{thm}

\smallskip

\section{Discussion and  outline of proofs}

\subsection{ Discussion}\label{L11} 
 
\begin{rmq} {\bf(Structure of the profile)}  
The global blow-up profile \eqref{infer1} 
 possesses a universal structure, given by the ``resolvent'' $G^{-1}$ of the ODE composed with the 
  universal time-space building block 
	\begin{equation}\label{block}
	T-t+\frac{|x|^2}{4|\log|x|^2|},
	\end{equation}
and the latter is not affected by the slowly varying factor $L$ of the nonlinearity.
At principal order, the effect of the latter on the profile is reflected only in the application 
of the resolvent $G^{-1}$ to this building block.  

A similar observation was made in \cite[Remark $2.1$]{chso} in the case of nonlinearities of the form $f(u)=u^pL(u)$.
In turn, we remark that the  time-space building block  \eqref{block} is the limit of that  in \cite[formula 
 $(2.10)$]{chso} (in the  power case) as $p\to \infty$. 
 
The above also applies for the final and refined space-time profiles and
for the sharp behavior in backward space-time parabolas
(cf. Corollary \ref{coro} and Theorem \ref{theo}).
\end{rmq}

\begin{rmq} {\bf(Examples)}  \label{ex1}
Typical examples of nonlinearities $f(s)$ to which Theorems \ref{Thm}-\ref{theo} apply are given by:
\be{examplesSlowVar}
\left\{\begin{aligned}
&\hbox{\ $\bullet$\  $(s+K)^q e^s$ for $K>0$ and $q\in\R$, or $K=0$ and $q\ge 1$,} \\
\noalign{\vskip 1mm}
&\hbox{\ $\bullet$\ $[\log^q(s+K)]e^s$ for $K>1$ and $q\in\R$, or $K=1$ and $q\ge 1$} \\
 \noalign{\vskip 1mm}
 &\hbox{\ $\bullet$\  $\exp(s\pm (s+1)^\nu)$ with $\nu\in (0,1/2)$,} \\
 \noalign{\vskip 1mm}
 &\hbox{\ $\bullet$\  the strongly oscillating functions} \\
 &\hbox{\qquad $(s+1)^{\sin[\log(s+1)]}e^s
\quad\hbox{ and }\quad
\exp\bigl[s+(s+1)^\nu\cos((s+1)^\gamma)\bigr],\ \ \nu, \gamma>0,\ \nu+\gamma<1/2$,} \\
\noalign{\vskip 1mm}
 &\hbox{\ $\bullet$\  $\bigl[1 +a \sin\bigl((s+1)^\nu\bigr)\bigr]e^s$ with $ \nu\in (0,1/2)$ and $|a|<1$.} \\
\end{aligned}
\right.
\ee
Indeed, by tedious but elementary computations, one can check  that the corresponding functions $L$
satisfy condition~\eqref{sem}.
Our results thus cover a large class of exponentially growing, non scale invariant nonlinearities.
However, it so far remains an open problem what is the largest possible class of $f$
for which the conclusions of Theorems \ref{Thm}-\ref{theo} hold.
\end{rmq}

\begin{rmq}  \label{remtypeI}
{\bf(Alternative type I assumption)}
For $n\le 2$, Theorems \ref{Thm}-\ref{theo} remain valid 
 (see Remark~\ref{remtypeIb} below)
if we replace the time increasing assumption $u_t\ge0$ by 
the type I blow-up assumption, namely
 \begin{equation}\label{dam1a}
	\|u(\cdot,t)\|_\infty\le G^{-1}(T-t)+ M,\quad T-\delta<t<T,
	\end{equation}
for some $M,\delta>0$.

In any dimension, property \eqref{dam1a} is in particular guaranteed if $u_t\ge 0$
(see~Lemma~\ref{lemtypeI}, which follows from arguments in \cite[Section 4]{friedman1985blow}).
However, in higher dimensions, the conclusions of Theorems \ref{Thm}-\ref{theo}
are not true in general under the weaker assumption \eqref{dam1a} (see Remark~\ref{remn39}(i)).

In the special case $f(u)=e^u$ with $u_0$ satisfying \eqref{Omega-u0}, \eqref{Omega-u0b}, 
property \eqref{dam1a} is also known to be true when $n=1$ 
(see~\cite[p.133]{HV93}) or $n\in[3,9]$ 
(see~\cite[Theorem~1]{fila2008nonconstant} and also \cite{souplet-rem}).
\end{rmq}

\subsection{Outline of proofs}

The route map of the proofs is organized in the following order,
where each step uses the previous ones:
\smallskip

{\bf $\bullet$ Step 1.} Convergence to the ODE solution in the scale $x=O(\sqrt{T-t})$

{\bf $\bullet$ Step 2.} Upper part of Theorem \ref{Thm} and lower estimate of $|u_r|$

{\bf $\bullet$ Step 3.} Theorem~\ref{theo}

{\bf $\bullet$ Step 4.} Lower part of Theorem \ref{Thm} 

{\bf $\bullet$ Step 5.} Corollary \ref{coro}

\smallskip

  Here, by the upper (resp., lower) part of Theorem~\ref{Thm}, we refer to
\begin{equation}\label{infer1-upper}
	u(x,t)\le G^{-1}\bigg(T-t+\frac{|x|^2}{4|\log|x|^2|}\bigg)+\phi_1(|x|,t)
	\end{equation}
 resp.,
\begin{equation}\label{infer1-lower}
	u(x,t)\ge G^{-1}\bigg(T-t+\frac{|x|^2}{4|\log|x|^2|}\bigg)+\phi_2(|x|,t)
	\end{equation}
	where $\phi_1,\phi_2\to 0$ as $(x,t)\to (0,T)$.

\medskip

Step 1 (Section~\ref{CVGexpon})
 is based on the following quasiconvergence property of $u$. 
  To this end we introduce the similarity profile equation
	\be{eqprofilez}
	z''+\Bigl(\frac{n-1}{r}-\frac{r}{2}\Bigr) z'+e^z-1=0,\quad r>0,\quad \hbox{with } z'(0)=0
	\ee
 and the relevant subset of equilibria
	\be{defS}
	\mathcal{S}:=\Bigl\{z\in C^2([0,\infty));\ z \text{ solution of \eqref{eqprofilez}} 
	\ \hbox{such that $z'\in L^\infty(0,\infty)$, $z'\le 0$ and $z(0)\ge 0$}\Bigr\}.
	\ee

 \begin{prop}\label{cvg}
 	Let $f\in C^1([0,\infty))$, with $f(0)\ge 0$ and
$f(s)=e^{s}L(e^{s})$, where $L$ satisfies 
	\be{sem11}
	  \frac{sL'(s)}{L(s)}=O\bigl(\log^{-\alpha}(s)\bigr)
	\hbox{ as $s\to\infty$, for some $\alpha>\frac12$.}
	\ee 
	Assume \eqref{Omega-u0}, \eqref{Omega-u0b},
	$T=T(u_0)<\infty$ and \eqref{dam1a} for some $M,\delta>0$.
	Then, for any sequence $t_j\to T$,  
	there exists a subsequence (still denoted $t_j$) and $\phi\in \mathcal{S}$ such that  
	\begin{equation}\label{jojo0}
	\lim_{j\to \infty} u(y\sqrt{T-t_j},t_j)-\psi(t_j)=\phi(|y|),
	 \hbox{ uniformly for y in compact sets of $\mathbb{R}^n$,}
	 	\end{equation}
	  where $\psi$ is defined in \eqref{defODE-sol}.
\end{prop}

 Proposition \ref{cvg} is proved by modifying arguments from \cite{chabi1},
which rely on transformation of the equation by similarity variables and ODE renormalization, along with a perturbed energy functional.
We recall that  an extension to non scale invariant, regularly varying nonlinearities 
 of the classical Giga-Kohn analysis \cite{giga1985asymptotically, giga1989nondegeneracy} for the pure power nonlinearity
 was done in \cite[Section 3]{chabi1}.
 We note that the similarity profile equation \eqref{eqprofilez}
 is the same as for the pure exponential. This is due to the fact that, under assumption \eqref{sem11},
 after appropriate ODE renormalization, 
 the terms coming from the slowly varying factor $L$ disappear in the limiting procedure (see Lemma~\ref{lem11}).
 We then use the assumption 
$u_t\ge 0$ of Theorem \ref{Thm} to eliminate all possible limits $\phi \in \mathcal{S}$ except $0$ 
(see Lemma \ref{cvg0}).
Since this assumption also guarantees 
 type I blowup (see Lemma \ref{lemtypeI}), this leads to the following consequence of Proposition \ref{cvg}.

\begin{prop}
\label{CVG} Under the assumptions of Theoren \ref{Thm}, we have 
		\be{1.9}
 		\lim_{t\to T} u(y\sqrt{T-t},t)-G^{-1}(T-t)=0\quad \text{ uniformly for $y$ bounded.}
 		\ee
\end{prop}

In Step 2, given in Section~\ref{upperrrrrr}, we establish the upper part of Theorem \ref{Thm}, along with a lower estimate of~$|u_r|$.
This will be derived  as a consequence of \cite[Theorem $3.1$]{chso} combined with Proposition \ref{CVG}.
The former, which follows from a maximum principle argument for a suitable auxiliary function,
provides a precise upper space-time estimate of $u$ in terms of $u(0,t)$, as well as a lower estimate of $|u_r|$,
 for radial decreasing solutions of \eqref{eqE1} with general nonlinearities $f(u)$, under appropriate assumptions on $f$.
Our task will be to check that these assumptions are satisfied for the exponential type nonlinearities in Theorem \ref{Thm}
and that the resulting estimates can be recast under the form required by Theorem \ref{Thm}.

\smallskip

Step 3 (cf.~Section~\ref{LOT5}) is the proof of Theorem~\ref{theo}.
 To this end, we further use similarity variables defined in \eqref{dam111} and ODE normalization by linearizing around $\psi$ as in \eqref{defw}
so as to describe more accurately the convergence in Proposition~\ref{cvg}.
For this, owing to the lower estimate of $|u_r|$ from Step 2 and since the linearized equation has the same form 
as in the case of nonlinearities $f(u)=u^pL(u)$, we can rely on a result from \cite{chso} 
for the linearized equation (see Proposition \ref{TheoremBreduit} below).
We note that the latter was based on modifications of the classical center manifold type analysis
introduced in \cite{filippas1992refined} (see also \cite{HV93, Vel93b, bebernes1992final,WXL})
and simplified in \cite{souplet2019simplified} in the case of radially decreasing solutions.

\smallskip

Steps 4 and 5 are given in Section~\ref{LOT4}.
We first prove the lower estimate of Theorem~\ref{Thm}.  To this end, 
for $t_0$ suitably close to $T$, we use the lower estimate on $u(t_0)$ provided by Theorem~\ref{theo}
in the region $|x|=O(\sqrt{T-t_0})$ and propagate it forward in time via a suitable comparison argument. 
Lastly, the final and refined space-time profiles in Corollary \ref{coro}
are easily deduced from the global profile in Theorem~\ref{Thm}.

\begin{rmq}\label{remn39}
(i) The result in Proposition \ref{cvg} is of independent interest since it is valid for more general solutions of \eqref{eqE1} 
( in particular, only type I blow-up rate is needed  instead of  $u_t\ge 0$), as compared with Theorem~\ref{theo}. But of course
 the latter gives more precise information on the asymptotic behavior of the solution. 
 Note that in dimensions $3\le n\le 9$, for the type I blowup solutions without the property $u_t\ge 0$,
  the conclusions of Theorems \ref{Thm}-\ref{theo} may fail.
Indeed, it is well known that $\mathcal{S}$ is not reduced to $\{0\}$ for $3\le n\le 9$
(it actually contains  infinitely many elements,
see e.g \cite{ET,fila2008nonconstant}),
 and such nontrivial $\phi\in\mathcal{S}$ provide radial decreasing, backward self-similar solutions of the equation 
 $u_t-\Delta u=e^u$ in $\mathbb{R}^n\times (0,T)$, of the form 
 $$ u(x,t)=\phi\Big(\frac{x}{\sqrt{T-t}}\Big)-\log(T-t),$$
 which exhibit different final profiles from those in Theorem \ref{Thm} and Corollary \ref{coro}.

\smallskip

 (ii) For $f(s)=e^s$, the result of Proposition \ref{CVG} was obtained in \cite{BE}.
The proof therein was based on similarity variables and energy arguments inspired from \cite{giga1985asymptotically}
(without perturbation), combined with an intersection-comparison analysis
with the singular steady-state $U^*$ of $\Delta u+e^u=0$.
 More precisely, the property $u_t\ge 0$ is used in \cite[Section~3]{BE} to show that 
 the solution (recast in similarity variables) intersects $U^*$ exactly once for $t$ close to $T$.
It is then shown in \cite[Section~4]{BE} by delicate ODE analysis that $0$ is the only element of $\mathcal{S}$ with such intersection property, 
and this allows to eliminate the nonzero limits $\phi\in\mathcal{S}$.
Here, by directly using the observation that the property $u_t\ge 0$ yields the additional constraint 
$r\phi'+2\ge 0$ on the possible limits $\phi\in\mathcal{S}$ in \eqref{jojo0}
(see \eqref{elim1}-\eqref{elim2}), it turns out that the nonzero limits
can be eliminated by using only a small part of the ODE analysis from \cite{BE}
(see Lemma \ref{cvg0} below). We thus avoid any use of intersection-comparison arguments, 
which leads to a simpler proof even in the special case $f(s)=e^s$.
\end{rmq}

\section{Proof of Propositions \ref{cvg} and \ref{CVG}} \label{CVGexpon}

Here and later on we will need various asymptotic properties of the ODE solution and of related quantities.
We start with the following.
In the rest of paper, the asymptotic notation $\sim$ means that the quotient of the two functions converges to $1$.

 \begin{lem}\label{lm1}  
Assume \eqref{2.2}, where $L$ is $C^1$ and positive for large $s$, and let $G$ be as in \eqref{f}.
  \smallskip
  
  (i) If
 		\begin{equation}\label{hyplm1}
 		\lim_{s\to\infty}\frac{sL'(s)}{L(s)}=0,
 		\end{equation}
 		then
 	\be{ch}
 	G(X)\sim \frac{e^{-X}}{L(e^X)}=\frac{1}{f(X)}, \quad X\to\infty,
 	\ee

 (ii) If $L$ satisfies \eqref{sem11}, then
	 		\begin{equation}\label{ch0}
	 		G^{-1}(Y)= -\log\Big(YL\Bigl(\frac{1}{Y}\Bigl)\Big)+o(1)\quad \hbox{as } Y\to 0^+
	 		\end{equation}
 and in particular 
 	 		\begin{equation}\label{ch0b}
			G^{-1}(Y)=(1+o(1))|\log Y|,\quad Y\to 0^+. 
				 		\end{equation}
	 \end{lem}
	  The proof of Lemma~\ref{lm1} (and of Lemma~\ref{lemconfG} below) makes use of the following property of $L$ (cf.~\cite[Lemma 7.2]{chso}), whose proof is recalled in Appendix for convenience.
	 \begin{lem}\label{ell8}
Under the assumptions of Proposition~\ref{cvg} on $L$ with $\alpha\in(0,1)$, there exists $s_1>0$ such that, for all $s>s_1$,
\begin{equation}\label{ell11}
	\Bigl|\frac{L(\lambda s)}{L(s)}-1\Bigl| \ \le 4\frac{|\log \lambda|}{\log^\alpha{\hskip -2.5pt}s},
	\quad\hbox{ for all $\lambda \in I_s:=\bigl[\exp(-\frac18\log^\alpha{\hskip -2.5pt}s),\exp(\frac18\log^\alpha{\hskip -2.5pt}s)\bigr].$}
	\end{equation}
\end{lem}

	 \begin{proof}[Proof of lemma \ref{lm1}]
	 (i)  For $X$ large enough, by the change of variable $\eta=e^\tau$, we have 
\be{controlesps}
G(X)=\int_X^\infty \frac{d\tau}{f(\tau)}=\int_X^\infty \frac{d\tau}{e^\tau L(e^\tau)}=\int_{e^X}^\infty \frac{d\eta}{\eta^2L(\eta)}=:Q(e^X).\ee
By \cite[Lemma 4.2(i)]{chso} with $p=2$, we have
$Q(s)\sim \frac{1}{sL(s)}$ as $s\to\infty$, hence \eqref{ch}.

	 	 \smallskip
		 (ii) 
	 	Using \eqref{ch} and setting $X=X(Y):=G^{-1}(Y)$, we write
\be{equivYbeta}
		 \frac{1}{Y}=e^{X}L(e^X) (1+\eps(e^X)), \quad\hbox{with $\lim_{s\to\infty} \eps(s)=0$.}
\ee
Let $\lambda(s):= (1+\eps(s))L(s)$.
Assuming $\frac{1}{2}<\alpha<1$ without loss of generality, it follows from \eqref{sem} that
$\log(L(s))=o(\log^{1-\alpha} {\hskip -2pt}s)$, hence $\log(\lambda(s))=o(\log^\alpha {\hskip -2pt}s)$ as $s\to\infty$.
  Applying Lemma~\ref{ell8}, we then deduce that
$L(s\lambda(s))\sim L(s)$ as $s\to\infty$.
Going back to \eqref{equivYbeta}, we obtain
$L(1/Y)=L(\lambda(e^{X})e^{X})\sim L(e^X)$ as $X\to\infty$, 
 and then 
$e^{-X}=YL(1/Y) (1+\eps_1(Y))$ with $\lim_{Y\to 0} \eps_1(Y)=0$,
hence \eqref{ch0}.
\end{proof}

\begin{proof}[Proof of Proposition \ref{cvg}] 
For clarity we split the proof into several steps. 
\smallskip

{\bf Step 1.} {\it Upper estimate of $|\nabla u|$.}

\begin{lem}\label{propo}	
	Under the assumptions of Proposition \ref{cvg}, we have
	\begin{equation}\label{pop}
	\|\nabla u(t)\|_\infty\le M_1 (T-t)^{-1/2}, \quad   t_0<t<T, 
	\end{equation}
	 with $t_0=T-\delta/2$ and $M_1=M_1(\Omega,f,M,T,\delta)>0$.
\end{lem}

\begin{proof}[Proof of Lemma \ref{propo}]
We first claim that, under the current assumptions,  
		\begin{equation}\label{cor}
		 \|f(u(\cdot,t))\|_\infty\le\frac{C(f,M,T)}{T-t},\quad T-\delta\le t<T.
		\end{equation}
 By \eqref{ch}, there exists $z_0>0$ such that 
\begin{equation}\label{A19091a}
0<f(X)\le \frac{2}{G(X)},\quad X\ge z_0.
\end{equation}
Moreover, there exists $z_1\ge z_0$ such that $L(e^{s+M})\le 2 L(e^s)$ for all $s\ge z_1$, hence 
\begin{equation}\label{A19091b}
G(X+M)=\int_{X+M}^{\infty}\frac{ds}{e^sL(e^s)}=\int_{X}^{\infty}\frac{ds}{e^{s+M}L(e^{s+M})}\ge  \frac12 e^{-M} G(X),
\quad X\ge z_1.
\end{equation}
Set $\Sigma:=\{(x,t)\in\Omega\times(T-\delta,T);\ u(x,t)\ge z_1+M\}$ and let $(x,t)\in \Sigma$.
By \eqref{dam1a} and the decreasing monotonicity of $G$, we have $G(u(x,t))\ge G(\psi(t)+M)$.
Moreover, $\psi(t)\ge \|u(t)\|_\infty-M\ge z_1$ hence 
$G(\psi(t)+M)\ge  \frac12 e^{-M} G(\psi(t))$ by \eqref{A19091b}.
Consequently, using \eqref{A19091a}, we obtain
\begin{equation}\label{A19091c}
0<f(u(x,t))\le \frac{2}{G(u(x,t))}\le \frac{2}{G(\psi(t)+M)}\le \frac{4e^M}{G(\psi(t))}= \frac{4e^M}{T-t},\quad (x,t)\in \Sigma.
\end{equation}
On the other hand, for $(x,t)\in (\Omega\times(T-\delta,T))\setminus\Sigma$, we have
\begin{equation*}
|f(u(x,t))| \le C(f,M) \le C(f,M)\frac{T}{T-t}.
\end{equation*}
This and \eqref{A19091c} prove claim \eqref{cor}.

\smallskip

 Now denote by $(e^{-tA})_{t\ge0}$ the heat semigroup in $\Omega$ (with Dirichlet condition for $\Omega\neq\mathbb{R}^n$) and
set $t_1=T-\delta$, $t_0=T-\delta/2$. Using the variation of constants formula, the gradient estimate   for the heat semigroup
(see, e.g.,~\cite[Proposition 48.7*]{quittner2019superlinear}),  \eqref{dam1a} and \eqref{cor}, 
we compute  for $t_0<t<T$,
	\begin{align*}
	\|\nabla u(t)\|_\infty&\le \|\nabla e^{-(t-t_1)A}u(t_1)\|_\infty+\int_{t_1}^{t}\|\nabla e^{-(t-\tau)A}f(u(\tau))\|_\infty d\tau\\
	&\le C(\Omega,T)(t-t_1)^{-1/2}\|u(t_1)\|_\infty+C(\Omega,T)\int_{t_1}^{t}(t-\tau)^{-1/2}\bigl\|f(u(\tau))\bigr\|_\infty d\tau\\
	&\le C(\Omega,T) \bigl(G^{-1}(\delta)+M\bigr)(T-t)^{-1/2}+C(\Omega,f,M,T) \int_{t_1}^{t}(t-\tau)^{-1/2}(T-\tau)^{-1} d\tau.
	\end{align*}
	 Setting $t^*=t-(T-t)$, we have  $t-\tau\ge \frac12(T-\tau)$ for $\tau\le t^*$,
	 hence
	\begin{align*}
	\int_{t_1}^{t}(t-\tau)^{-1/2}(T-\tau)^{-1} d\tau
		&\le C\int_{t_1}^{t^*}(T-\tau)^{-\frac32} d\tau+(T-t)^{-1}\int_{t^*}^{t}(t-\tau)^{-1/2}d\tau\\
	&\le C\bigl[(T-\tau)^{-\frac12}\bigr]_{t_1}^{t^*} +2(T-t)^{-1}(t-t^*)^{\frac12}
	\le C(T-t)^{-1/2},
	\end{align*}
	and \eqref{pop} follows.
\end{proof}

\smallskip

{\bf Step 2.} {\it Rescaled equation.}
 Let $(y,s)$ be the similarity variables defined as 
\begin{equation}\label{dam111}
y:=\frac{x}{\sqrt{T-t}},\quad s:=-\log(T-t).
\end{equation}
 By \eqref{ch0b},
we may take $s_0>0$ large enough so that $\psi_1(s):=\psi(T-e^{-s})$ exists and is $>\frac18s$ for $s\ge s_0$.
We define the rescaled function
\begin{equation}\label{dam11}
w(y,s):=u(y e^{-s/2},T-e^{-s})-\psi_1(s).
\end{equation}
Note that it is equivalent to $u(x,t)=\psi(t)+w(y,s)$  with $\psi$ as in \eqref{defODE-sol}. By a simple computation  using $\psi_1'=e^{-s}f(\psi_1)$, we have  
 \begin{equation}\label{7.3'}
\ w_s=-\frac{y}{2}\cdot\nabla w+e^{-s}u_t-e^{-s}f(\psi_1),
\end{equation}
and
$$\nabla w=e^{-s/2}\nabla u, \quad \Delta w=e^{-s}\Delta u,$$
hence $$w_s-\Delta w+\frac{1}{2}y\cdot \nabla w=e^{-s}\Bigl(u_t-\Delta u-f(\psi_1)\Bigl).$$

Then $w$ is a global (in time) solution of 
\begin{equation}\label{dam2}
w_s-\Delta w +\frac{1}{2}y\cdot\nabla w= h(s)\Big(e^w\frac{L(e^{ \psi_1+ w})}{L(e^{\psi_1})}-1\Big),\quad (y,s)\in\mathcal{W},
\end{equation}
where
\begin{equation}\label{LO}
\begin{aligned}
\mathcal{W}&:=\{(y,s) : s_0<s<\infty, y\in D(s)\}\quad \text{with } D(s):= e^{s/2}\Omega\\&\text{and}\quad
h(s):=e^{-s}f(\psi_1)=e^{-s+\psi_1}L(e^{\psi_1})
\end{aligned}
\end{equation} 
 (note that $h\equiv 1$ for $L\equiv 1$). Observe that $\lim_\infty D(s)=\mathbb{R}^n$ and that $\mathcal{W}=\mathbb{R}^n\times(s_0,\infty)$  for $\Omega=\mathbb{R}^n$. In term of the variables $y$ and $s$ and rescaled function $w$, \eqref{dam1a} and Lemma \ref{propo} imply 
 \begin{equation}\label{wysM}
 w\le M,
\end{equation}
\begin{equation}\label{DwysM1}
|\nabla w|\le M_1,
\end{equation}
where $M_1=M_1(\Omega,M,f,T)$.  By the proof of \cite[Proposition 23.1]{quittner2019superlinear} and the radial decreasing property, we have $u(0,t)=\|u(t)\|_\infty\ge \psi(t)$ hence $w(0,s)\ge 0$ for $s>s_0$, and \eqref{wysM}-\eqref{DwysM1} imply
 \be{L12}
 -M_1|y|\le w(y,s)\le  M.
 \ee 
 
We  rewrite equation \eqref{dam2} as 
\begin{equation}\label{VraieE0}
 w_s-\Delta w +\frac{1}{2}y\cdot\nabla w=e^w-1 + H(s,y)
\end{equation}
which is equivalent to
\begin{equation}\label{VraieE}
\rho w_s-\nabla\cdot(\rho\nabla w)=(e^w-1)\rho +\rho H(s,y)
\end{equation}
where  the nonautonomous term is given by
\begin{equation}\label{VraieEE}
H(s,y):=\big(h(s)-1\big)\big(e^{w}-1\big)+h(s)e^w\bigg(\frac{L(e^{\psi_1+w})}{L(e^{\psi_1})}-1\bigg).
\end{equation}

\smallskip

{\bf Step 3.} {\it Control of nonautonomous terms}.
 It is clear that $H\equiv 0$ when $L\equiv 1$. 
 Now, under assumption \eqref{sem11}, we shall suitably control the convergence of $h(s)$ to $1$
 and control the nonautonomous term $H$, in a way that will ensure the existence of a Liapunov functional.

\begin{lem}\label{lem11} 
 Under the assumptions of Proposition \ref{cvg}, there exist $C_0>0$ and $s_1\ge s_0$ such that
	 \begin{equation}\label{limhbeta}
	 |h(s)-1|\le C s^{-\alpha},\quad s>s_1
	 \end{equation}
and
	\begin{equation}\label{decayH}
	\|H(s,\cdot)\|_\infty\le C_0 s^{-\alpha}\log s,\quad s>s_1.
	\end{equation}
\end{lem}

\begin{proof}

Integrating by parts, we have
$$
G(X)=\int_{X}^{\infty}\frac{dz}{f(z)}=\int_{X}^{\infty}\frac{e^{-z} dz}{L(e^z)}
=\frac{e^{-X}}{L(e^X)}-\int_{X}^{\infty}\frac{L'(e^z)}{L^2(e^z)}dz,
$$
hence
\be{XpLF}
\bigl|e^{X}L(e^X)G(X)- 1\bigr|
\le Ce^{X}L(e^X)\int_{X}^{\infty} \frac{e^{-z}}{\log^\alpha{\hskip -2pt}(e^z) \, L(e^z)}dz
\le C\frac{e^{X}L(e^X)G(X)}{ X^\alpha},\quad X\ge 2.
\ee
In particular $\displaystyle\lim_{X\to\infty} e^{X}L(e^X)G(X)=1$ and then  we get
$$\bigl|e^{X}L(e^X)G(X)- 1\bigr|\le CX^{-\alpha},\quad X\ge 2.$$
Since $G(\psi_1(s))=G(\psi(T-e^{-s}))=e^{-s}$, we have $h(s)=G(\psi_1(s))e^{\psi_1(s)}L(e^{\psi_1(s)})$.
 Since also $\psi_1(s)=G^{-1}(e^{-s})\sim s$ owing to \eqref{ch0b},
property \eqref{limhbeta} follows.
\smallskip

 Let us next prove \eqref{decayH}.
We write $H(s,y)\equiv H_1(s,y)+H_2(s,y)$  according  to \eqref{VraieEE}. Using  \eqref{wysM} and \eqref{limhbeta},
	we  have  $$|H_1(s,y)|:=\bigl|(h(s)-1)\big(e^w-1\big)\bigr|\le \frac{C}{s^\alpha},$$
	which shows that $H_1$ satisfies the desired property for all $s>s_0$. Now we shall show that $H_2$ also satisfies \eqref{decayH}. We proceed  similarly 
	as in \cite[Lemma $3.1$]{chabi1}.
	Taking $s_1>s_0$ sufficiently large, we have
	\begin{equation}\label{7.12}
	e^{\psi_1(s)}\ge e^{s/4},\quad s\ge s_1.
	\end{equation}
	Let $Q=\mathcal{W}\cap\Bigl(\mathbb{R}^n\times [s_1,\infty)\Bigl)$ and $E=\bigl\{(y,s)\in Q;\ e^{w(y,s)}\le s^{-\alpha}\bigr\}$. By  \eqref{limhbeta} and the slow variation  property of $L$, we have
	\begin{equation}\label{(3)}
	|H_2(s,y)|\le C_0 e^w\le C_0 s^{-\alpha},\quad (y,s)\in E.
	\end{equation}
	Next consider the case when $(y,s)\in Q\setminus E$. 
	Then, taking $s_1$ large enough, we have $e^{w+\psi_1}\ge s^{-\alpha}e^{s/4}\ge e^{s/8}\ge 2$. 
	By assumption \eqref{sem11},
	$$\Sigma(X):=\sup_{z\ge X} \Bigl|{zL'(z)\over L(z)}\Big|\le C (\log X)^{-\alpha},\quad X\ge 2.$$
	Therefore, $\Sigma(e^{s/8})\le C (\log(e^{s/8}))^{-\alpha}\le  C s^{-\alpha}$, as well as $M\ge e^w\ge s^{-\alpha}$. 
	 Using \eqref{7.12}, it follows that
	\begin{align*}\Bigl|\log\Bigl({L(e^{w+\psi_1})\over L(e^{\psi_1})}\Bigr)\Bigr|=\Bigl|\int_{e^{w+\psi_1}}^{e^{\psi_1}} {L'(z)\over L(z)} dz\Bigr|
	&\le \Sigma\Bigl(\min\{e^{w+\psi_1},e^{\psi_1}\}\Bigl)\Bigl|\bigl[\log z\bigr]_{e^{w+\psi_1}}^{e^{ \psi_1}}\Bigl|\\
	&\le \Sigma\Bigl(e^{s/8}\Bigl) |\log e^{w}|
	\le C s^{-\alpha}\log s.
	\end{align*}
	Hence,
	$$|H_2(s,y)|\le 2M\Bigl|{L( e^{ w+\psi_1})\over  L(e^{\psi_1})}-1\Bigr|\le C s^{-\alpha}\log s.$$
	Combining with \eqref{(3)}, we deduce that
	$$\|H_2(s,\cdot)\|_\infty\le Cs^{-\alpha}+C s^{-\alpha}\log s\le Cs^{-\alpha}\log s,\quad s\ge s_1.$$
\end{proof}

\smallskip

{\bf Step 4.} {\it Weighted energy functional and its properties.}
The Liapunov functional for equation  \eqref{VraieE} will be given by the    weighted energy  defined as follows: 
\begin{equation}\label{LO2}
\E[w](s):= E[w](s)+C_1 s^{-\gamma}
\end{equation}
where  $\gamma=\alpha-\frac12$, $$E[w]:=\int_{D(s)}\bigg(\frac{1}{2}|\nabla w|^2+w-e^w\bigg)\rho dy,$$
 and  $C_1>0$ is a  constant.

\begin{lem}\label{ell7}
	For all $s>s_0$, there exists $M_0>0$ (independent of $s,y$) such that 
	\begin{equation}\label{jojo1}
	\frac{d}{ds}\E[w](s)\le-\frac{1}{2}\int_{D(s)}w_s^2\rho\le 0\end{equation}
	
	and 
	
	\begin{equation}\label{JO1}
	\E[w](s)\ge -M_0.
	\end{equation}
	
\end{lem}

\begin{proof}
	Fix any $s_1>s_0$.  By parabolic regularity (see, e.g., \cite[Lemma $3.2.$]{chabi1} for details), we have
	\begin{equation}\label{jo2}
	D^2w,\ (1+|y|)^{-1}w_s,\ (1+|y|)^{-1}\nabla w_s\in BC(\overline{\mathcal{W}}_{s_1}),
	\end{equation}
	where $\mathcal{W}_{s_1}=\mathcal{W}\cap(\R^n\times(s_0,s_1))$ and $BC$ denotes the set of bounded continuous functions. Using the exponential decay of $\rho$, this guarantees the convergence and the differentiability of the
	various integrals and justifies the integrations by parts in the rest of the proof.
	
First, we compute the variation of the first part $E[w](s)$ of the energy.	We have
	\begin{equation}\label{Loth1}
	\begin{aligned}
	\frac{d}{ds}\int_{D(s)} w\rho dy&=\int_{D(s)}\rho w_sdy+\frac{1}{ 2}\int_{\partial D(s)}w\rho (y\cdot\nu)d\sigma,\\
	\frac{d}{ds}\int_{D(s)}e^w\rho dy&=\int_{D(s)}\rho w_se^wdy+\frac{1}{ 2}\int_{\partial D(s)}e^w\rho (y\cdot\nu)d\sigma,\\
	\end{aligned}
	\end{equation}
	 where $d\sigma$ denotes the surface measure on $\partial D(s)$ 
	and $\nu$ the exterior unit normal on $\partial D(s)$.
	
	 Now, we control the variation of the term involving $\nabla w$ by using  integration by parts as follows
		$$\begin{aligned}
		\frac{d}{ds}\int_{D(s)}|\nabla w|^2\rho
		&=2\int_{D(s)} \nabla w_s\cdot(\rho\nabla w)+\frac12\int_{\partial D(s)}|\nabla w|^2\rho(y\cdot\nu)d\sigma\\
		&=-2\int_{D(s)} w_s\nabla\cdot(\rho\nabla w)+2\int_{\partial D(s)}\rho w_s(\nabla w\cdot\nu) d\sigma+\frac12\int_{\partial D(s)}
		\rho|\nabla w|^2(y\cdot\nu)d\sigma,
		\end{aligned}$$
		and then use
		$w_s=-\frac{y}{2}\cdot\nabla w-h(s)$ (owing to \eqref{7.3'}) and $\nabla w=(\nabla w\cdot\nu)\nu$ 
		on $\partial D(s)$, hence
		$w_s(\nabla w\cdot\nu)=-\frac12|\nabla w|^2(y\cdot\nu)-h(s)\nabla w\cdot\nu$, which yields
		\begin{equation}\label{LDC1303242b}
		\frac{d}{ds}\int_{D(s)}|\nabla w|^2\rho=-2\int_{D(s)} w_s\nabla\cdot(\rho\nabla w)
		-\frac12\int_{\partial D(s)}\rho|\nabla w|^2(y\cdot\nu)d\sigma-2h(s)\int_{\partial D(s)}\rho\nabla w\cdot\nu d\sigma.
		\end{equation}
		Also, we have $\partial D(s)=e^{s/2}\partial\Omega$ hence $\int_{\partial D(s)}d\sigma\le C(\Omega)e^{s/2}$ and $|y|= Re^{s/2}$ on $\partial D(s)$. 
		Consequently, using \eqref{DwysM1}, we get (here and in the rest of the proof, $C$ denotes a generic positive constant independent of $s, y$)
		\begin{equation}\label{Loth2}
		\begin{aligned}
		\int_{\partial D(s)}\rho|\nabla w|^2|y\cdot\nu|d\sigma\le& CM_1^2e^{s/2}\exp\bigl[-(Re^{s/2})^2/4\bigr]\le C\exp\bigl(-ce^s\bigr),\\
		\int_{\partial D(s)}\rho|\nabla w\cdot\nu| d\sigma\le& CM_1e^{s/2}\exp\bigl[-(Re^{s/2})^2/4\bigr]\le C\exp\bigl(-ce^s\bigr),
		\end{aligned}
		 \end{equation}
		 and using $w=-\psi_1$ on $\partial D(s)$  we have 
		 \begin{equation}\label{Loth3}
		 \Big|\int_{\partial D(s)}w\rho (y\cdot\nu)d\sigma-\int_{\partial D(s)}e^w\rho (y\cdot\nu)d\sigma\Big|\le C \psi_1(s)e^{s/2}\exp\bigl[-(Re^{s/2})^2/4\bigr]\le C\exp\bigl(-ce^s\bigr) .
		 \end{equation}
		 
		 Combining \eqref{VraieE} and \eqref{Loth1}-\eqref{Loth3}, we obtain
		 \begin{equation}
		 \begin{aligned}
		 \frac{d}{ds}E[w]&\le -\int_{D(s)} w_s^2\rho+\int_{D(s)}w_sH(s,y)\rho\strut+C\exp\bigl(-ce^s\bigr)\\
		 &\le -\frac{1}{2}\int_{ D(s)}w_s^2\rho+\frac{1}{2}\int_{D(s)}H^2(s,y) \rho 
		 +C\exp\bigl(-ce^s\bigr)\label{Loth4}
		 \end{aligned}
		 \end{equation}
and \eqref{decayH} and $\alpha>\frac12$ guarantee that, for some $C_1>0$,
  \begin{equation}\label{ddpsi2}
  \frac{1}{2}\int_{D(s)}H^2(s,y)\rho dy+C\exp\bigl(-ce^s\bigr)
  \le \gamma C_1 s^{-\alpha-\frac12}
  \quad\hbox{for all}\quad s\ge s_0.
  \end{equation}
  It then follows from  \eqref{LO2} and \eqref{Loth4} that
  $$\frac{d}{ds}\E[w](s)=\frac{d}{ds}E[w](s)-\gamma C_1 s^{-\alpha-\frac12}\le -\frac{1}{2}\int_{ D(s)}w_s^2\rho<0.$$	
	Combining  estimate \eqref{L12} with exponential decay of $\rho$ and boundedness of $|\nabla w|$  we obtain \eqref{JO1}.
\end{proof}

\smallskip

{\it{\bf Step 5.} Convergence and conclusion.}
	 By \eqref{jojo1}-\eqref{JO1}, we have
	\begin{equation}\label{ellGw}
	 \ell:=\lim_{s\to\infty} \E[w](s)\in [-M_0,\infty).
	\end{equation}
	  Pick any sequence $s_j\to \infty$ and set $z_j:=w(y,s+s_j)$.
Since $-c|y|\le z_j(y,s)\le \tilde{M}$ and $|\nabla z_j(y,s)|\le M_1$ (due to \eqref{L12} and \eqref{DwysM1}),
	it follows from 
\eqref{VraieE0}, \eqref{decayH} and parabolic estimates that
	the sequence $(z_j)_j$ is precompact in $C^{2,1}_{loc}\big(\mathbb{R}^n\times[0,1]\big)$.
	Consequently, there exists a subsequence (still denoted $s_j$) and a solution $\phi$ of
\begin{equation}\label{Vraie00}
 w_s-\Delta w +\frac{1}{2}y\cdot\nabla w=e^w-1 \quad \hbox{in } \mathbb{R}^n\times[0,1]
\end{equation}
such that $w(\cdot,\cdot+s_j)\to \phi$ in $C^{2,1}_{loc}\big(\mathbb{R}^n\times[0,1]\big)$. 
	Moreover 
$|\nabla \phi|$ is bounded in $\mathbb{R}^n\times[0,1]$ and  $-c|y|\le \phi(y,s)\le \tilde{M}$. Let $ K>0$. 
	 Using \eqref{jojo1}, \eqref{ellGw} and the fact that $B_{ K}\subset D(s)$ for all sufficiently large $s$, we deduce that
	\begin{equation*}
	\int_{0}^{1}\int_{B_{ K}}\big(\partial_s z_j\big)^2\rho dy ds\le 
	 \int_{s_j}^{s_j+1}\int_{D(s)}\big(\partial_s w\big)^2\rho dyds
		\le 2\E[w](s_j)-2\E[w](s_j+1)\to 0,\quad j\to \infty.
	\end{equation*}
	By Fatou's lemma, and since $ K>0$ is arbitrary,
	we deduce that $\partial_s \phi\equiv 0$ in $B_{ K}$. Therefore, by \eqref{Vraie00}, 
	we have $\Delta \phi-\frac12 y\cdot\nabla \phi=1-e^\phi$ in $\mathbb{R}^n$.  
	Moreover, since  $u$ is radially symmetric, so is $\phi$. 
In terms of  the variables $(x,t)$ and $u$, in view of \eqref{dam111} and \eqref{dam11}, we have $w(y,s_j)=u(y\sqrt{T-t_j},t_j)-\psi(t_j)$, which concludes the proof.
\end{proof}
\smallskip

Proposition  \ref{CVG} follows from Proposition \ref{cvg}  and the following two lemmas. 
The first one shows that there are no nontrivial radially decreasing backward self-similar profiles
that correspond to a time increasing solution.
This lemma is essentially due to \cite{BE} ( under slightly different form; see~\cite[Lemma 4.2(a)]{BE}).
We give a proof in Appendix  (significantly simpler than that in~\cite{BE}; moreover we remove the boundedness assumption on $\phi'$.).

\begin{lem}\label{cvg0}
 Let  $c(r):=\frac{n-1}{r}-\frac{r}{2}$ and $F(s):=e^s -1$.
Let $\phi$ be a solution of 
\begin{equation}\label{ODEphi}
\begin{cases}
\phi''+c(r)\phi'+F(\phi)=0,\quad 0<r<\infty\\
\phi(0)\ge0, \ \phi'(0)=0
\end{cases}
\end{equation}
such that  $\phi'\le 0$.
If  $g(r):=1+\frac12 r\phi'(r)\ge0$ for all $r\ge 0$ then $\phi\equiv0$.
\end{lem}

The second lemma shows that the assumption $u_t\ge 0$ implies the type I blowup property \eqref{dam1a}.
The proof is based on the argument in \cite{friedman1985blow} (see also \cite{BE}),
combined with the properties of $f$ and the asymptotics of the ODE solution obtained in Lemma \ref{lm1}.

\begin{lem}\label{lemtypeI}
Under the assumptions of Theoren \ref{Thm}, estimate \eqref{dam1a} is satisfied.
\end{lem}

\begin{proof}
Under assumptions \eqref{sem0}-\eqref{sem}, we have in particular
\be{limLprime}
		\lim_{s\to\infty}\frac{sL'(s)}{L(s)}=\lim_{s\to\infty}\frac{s^2L''(s)}{L(s)}=0
		\ee
	 and there exists $ A>0$ such that 
		\be{fposM}
		 f\in C^2([A,\infty))\quad\hbox{and}\quad f,f',f''>0\ \hbox{ on $[A,\infty)$}
		\ee 
		and $G:[A,\infty)\to(0,G(A)]$
		is decreasing and $C^2$.
Now, since $\lim_{t\to T}\|u(t)\|_\infty=\infty$, there exists $(x_0,t_0)\in B_R\times (0,T)$ such that $u(x_0,t_0)\ge A$  
and, setting $R'=|x_0|$ and using $u_r\le 0$ and $u_t\ge0$, we have 
		\be{fposM2}
		u(x,t)\ge A, \quad (x,t)\in Q:=\bar B_{R'}\times[t_0,T).
		\ee
 Set $J=u_t-\delta f(u)$, with $\delta>0$ to be determined. By direct calculation, we have 
 $$J_t-\Delta J-f'(u)J=f''(u)|\nabla u|^2\ge0 \quad \hbox{ in $Q$}.$$
 On the other hand, as a consequence of Lemma~\ref{BU1}(ii) below (which is proved independently),
$0$ is the only blow-up point of $u$.
 Therefore, $f(u)\le C_0$ and $u_t\ge C>0$ on $\partial_p Q$,
 which implies that $J>0$ in $\partial_p Q$ for some sufficiently small $\delta>0$. By the maximum principle, we deduce that $J>0$ in $Q$. An integration yields the following upper bound $$u(0,t)\le  G^{-1}(\delta(T-t))\quad \hbox{in } [t_0,T).$$
Using the radial decreasing property of $u$ and the expansion \eqref{ch0} of $G^{-1}$ in Lemma \ref{lm1}, we obtain \eqref{dam1a}.
\end{proof}

\begin{proof}[Proof of Proposition  \ref{CVG}] 
 The assumptions of Proposition \ref{cvg} are verified (including \eqref{dam1a} owing to Lemma
\ref{lemtypeI}).
It thus suffice to prove that 
$\phi\equiv 0$ is the only possible limit in \eqref{jojo0}
under the assumptions of Theorem \ref{Thm}.
Using the notation of 
 Step 5 of the proof of Proposition \ref{cvg}, we have 
$$\lim_{j\to \infty} w_s(y,s+s_j)=0,\quad (y,s)\in \R^n\times [0,1]$$
and it follows from \eqref{7.3'} that 
\be{elim1}
w_s(y,s+s_j)+\frac{y}{2}\cdot \nabla w(y,s+s_j)+h(s+s_j)\ge0,\quad (y,s)\in \R^n\times [0,1].
\ee
Since $\lim_{s\to \infty}h(s)=1$ by \eqref{limhbeta}, we deduce that $\phi\in \mathcal{S}$ and satisfies 
\be{elim2}
1+\frac{r}{2}\phi'(r)\ge 0, \quad r\ge 0.
\ee
Lemma~\ref{cvg0} then garantees 
that $\phi\equiv 0$, which concludes 
the proof.
\end{proof}

 \section{Proof of upper part of Theorem \ref{Thm}} \label{upperrrrrr}
   
The upper estimate of $u$ in Theorem \ref{Thm} i.e. \eqref{infer1-upper} will be a consequence of the following result from \cite{chso} (see \cite[Theorem $3.1$ and formula (4.22)]{chso}). 

\begin{thm}\label{rappelH}
Let $\Omega=\mathbb{R}^n \text{ or }\Omega=B_R$, and $u_0\ge 0$ be radially symmetric, nonincreasing in $|x|$ and  $T<\infty$.
Assume that  
\begin{equation*}
 f\in C^1([0,\infty)),\quad f(0)\ge 0, \quad \lim_{s\to \infty}f(s)=\infty,
 \end{equation*}
 \begin{equation*}
	\hbox{$f$ is $C^2$ and } \frac{f}{\log f} \hbox{ is convex at $\infty$,} \quad \int^\infty \frac{\log f(s)}{f(s)}ds <\infty.
	\end{equation*}
	If $\Omega=\mathbb{R}^n$ assume also that $u_0$ is nonconstant. 
	Then there exist $A_0,\rho>0$ such that
	\begin{equation}\label{upperTheoL}
	u(x,t)\le H^{-1}\bigg(H(u(0,t))+\frac{|x|^2}{4}\bigg)\quad \text{in}\ B_\rho\times(T-\rho,T),
	\end{equation}
	where
	$$ 
	H(X)=\int_X^\infty \frac{A_0+\log f(s)}{f(s)}ds.
	$$
Moreover, there exist constants $A,M>0$ such that 
\be{Rappelupper}
-u_r\ge \frac{rf(u)}{2(A+\log f(u))},\quad \hbox{in } Q:=\big\{(x,t)\in B_{R/2}\times(T/2,T); \ u(x,t)\ge M\big\}.
\ee
\end{thm}

\smallskip

 We first apply Theorem \ref{rappelH}, establishing the required properties of $f$.
This is the contents of the following lemma.

 \smallskip

\begin{lem}\label{BU1}
 (i) Let $f(s)=e^sL(e^s)$ satisfy \eqref{sem0} and  \eqref{limLprime}.
Then, we have
	\begin{equation}
	 \lim_{s\to\infty} f(s)=\infty
	\end{equation}
	and 
	\begin{equation}\label{Lo}
	\frac{f}{\log f}\quad \text{is convex at } \infty\   \text{and } \int^{\infty}\frac{\log f}{f}<\infty.
	\end{equation}
	
 (ii) Under the assumptions of Theorem \ref{Thm}, properties \eqref{upperTheoL} and \eqref{Rappelupper} are satisfied.
\end{lem}

 \begin{proof}
 (i) Since for any $\mu>0$ there exists $s_0$ such that $s^\mu  L(s)\ge \frac12$ for all $s\ge s_0$, we have $\lim_{s\to \infty} f=\infty$ and  there exists $M>0$ such that $f>e^4$ and $0<1/\log f<\frac{1}{4}$ in $[M, \infty)$ and $f\in C^2(M,\infty)$.  Set $F=\frac{f}{\log f}$, we shall show that
\begin{equation}\label{supo}
F''\ge 0\quad \hbox{in } [M,\infty).
\end{equation} 
 We have  $f'/f=1+\theta (e^s)$, where $\theta(X)=\frac{XL'(X)}{L(X)}$, then 
\begin{align*}
F'(s)=&\frac{f'}{\log f}-\frac{f'}{(\log f)^2}=\Psi F,\ \hbox{where }\Psi(s)=(1+\theta(e^s))\big(1-\frac{1}{\log f(s)}\big)
\end{align*}
and 
$\frac{F''}{F}=\Psi'+\Psi^2.$ 
Since $$\Psi'(s)=e^s\theta'(e^s)(1-\frac{1}{\log f(s)})+(1+\theta(e^s))^2\frac{1}{\log f(s)}$$
 and \eqref{limLprime} implies $\lim_{X\to \infty} \theta(X)=\lim_{X\to \infty}X\theta'(X)=0$, we deduce that $\lim_{s\to \infty}\Psi(s)=1$ and $\lim_{s\to \infty}\Psi('s)=0$, so that $\lim_{s\to \infty}\frac{F''}{F}=1$, hence \eqref{supo}.

For the second assertion of \eqref{Lo}, using that for any $\mu>0$ there exists $X_0>0$ such that $s^\mu L(s)$ is increasing and $\log f(s)\le 2s$ on $[X_0,\infty)$,  we have (with $\mu=\frac14$)
$$
\int^{\infty}\frac{\log f}{f}=\int^{\infty}\Big(\frac{\log f(s)}{e^{s/4}}\Big)\Big(\frac{1}{e^{s/2}}\Big)\Big(\frac{1}{e^{s/4}L(e^s)}\Big) ds\lesssim \int^{\infty}\frac{1}{e^{s/2}}ds<\infty.
$$ 
Hence  \eqref{Lo}.

 \smallskip
 
 (ii) This is a consequence of assertion (i) and of Theorem \ref{rappelH}.
\end{proof}

\smallskip

 Building on property \eqref{ch}, we next relate the asymptotics of $H$ and $G$ and their inverses, 
making use of the slow variation assumption $sL'(s)=o(L(s))$ as $s\to \infty$. 

 \begin{lem}\label{lem1}  
 	Let $f(s)=e^sL(e^s)$ and assume that $L$ is $C^1$ and positive for large $s$ and satisfies \eqref{hyplm1}.
	\smallskip
	
	 		(i) As $X\to\infty$, we have
 	\begin{align} 
 	\label{ch1}
 	H(X)&\sim \frac{Xe^{-X}}{L(e^X)}=\frac{X}{f(X)}, \\ 
 	\label{ch2}
 	H(X)&\sim G(X)|\log(G(X))|. 
 	\end{align}
	
		\smallskip
		
 	 (ii) If a function $\eps$ satisfies $\lim_{X\to\infty}\eps(X)=0$
 		(resp., $\lim_{Y\to 0^+}\eps(Y)=0$) then
 		\begin{equation}\label{equivGG}
 		\lim_{X\to\infty}\frac{G\bigl(X+\eps(X)\bigr)}{G(X)}=1
 		\quad\Bigg(\hbox{resp.,} \lim_{Y\to 0^+} \Bigg\{G^{-1}\Big[(1+\eps(Y))Y\Big]-G^{-1}(Y)\Bigg\}=0\Bigg)
 			\end{equation}
 		and the similar properties hold for $H, H^{-1}$.
 		
			\smallskip
			
 	(iii) We have
  	\begin{equation}\label{dam}
	\lim_{Y\to 0^+} \Bigg(H^{-1}(Y)- G^{-1}\bigg(\frac{Y}{|\log Y|}\bigg)\Bigg)=0.
	\end{equation}
 	\end{lem}

	\begin{proof}[Proof of Lemma \ref{lem1}]  
	We may fix $s_0>0$ such that $L(e^s)>0$ for all $s\ge s_0$.

 \smallskip
 
(i) Let $X>s_0$.
Since $\log L(\eta)=o(\log\eta)$ as $\eta\to\infty$ by \eqref{hyplm1}, we have
$$H(X)\sim\int_X^\infty \frac{\log f(s)}{f(s)}ds=
\int_X^\infty \frac{s+\log L(e^s)}{e^s L(e^s)}ds
\sim \int_X^\infty \frac{s}{e^s L(e^s)}ds,$$
and \eqref{ch1} follows from \eqref{ch} applied with $\tilde L(s)=\frac{L(s)}{\log s}$.
 Using \eqref{ch} again, which implies $|\log G(X)|\sim  X$ as $X\to\infty$,
we deduce \eqref{ch2} from \eqref{ch1}. 
 \smallskip
 
(ii) Let $\epsilon\in (0,1/2)$. Since $L$ has slow variation at $\infty$, we easily deduce from \eqref{ch} the existence of $\eta,X_0>0$ such that
$$G(X+\eta)\ge(1-\epsilon)G(X)\quad\hbox{and}\quad G(X-\eta)\le(1+\epsilon)G(X)\quad\hbox{for all $X\ge X_0$.}$$
Since $G$ is decreasing at $\infty$, this implies property \eqref{equivGG} for $G$.

For $Y\in(0,Y_1]$ with $Y_1>0$ sufficiently small, using \eqref{controlesps} we have $$G^{-1}(Y)=\log Q^{-1}( Y).$$ 
This and \cite[formula (4.32)]{chso} for $p=2$ imply property \eqref{equivGG} for $G^{-1}$.
Since $H$ is the function $G$ corresponding to $\tilde L(\tau)=\frac{L(\tau)}{A_0+\log \tau+\log L(\tau)}$
and $\tilde L$ satisfies \eqref{hyplm1}, the similar properties hold for $H, H^{-1}$.

 \smallskip
 
 (iii) For all $Y>0$ sufficiently small we may set $X=G^{-1}\bigl({Y\over |\log Y|}\bigr)$ ($\to\infty$ as $Y\to 0^+$),
so that ${Y\over |\log Y|}=G(X)$.
As $Y\to 0^+$, we have $\log G(X)\sim \log Y$ so that, owing to \eqref{ch2},
    $$Y=G(X)|\log Y|\sim G(X)|\log G(X)|\sim H(X).$$
    Applying property \eqref{equivGG} for $H^{-1}$, we get \eqref{dam}.
	\end{proof}
	
	 Combining the results in Lemmas \ref{BU1}(ii) and \ref{lem1}, and also making use of
	property \eqref{1.9} in Proposition~\ref{CVG} at $x=0$, we are now in a position to complete the proof.

 	\begin{proof}[Proof of  \eqref{infer1-upper}]
 Since $m(t):=u(0,t)=\|u(t)\|_\infty\to \infty$, we have  $G(m(t))\to0^+$ 
 as $t\to T$. Applying property \eqref{ch2} to  inequality \eqref{upperTheoL}, 
 we get
$$u(x,t)\le H^{-1}\left(H(m)+\frac{|x|^2}{4}\right)
 =H^{-1}\left((1+o(1))\left(G(m)|\log(G(m))|+\frac{|x|^2}{4}\right)\right)$$
as $(x,t)\to (0,T)$.  Next applying Lemma~\ref{lem1} (ii) and (iii),
we obtain
 \begin{equation}\label{nuitt}
 u(x,t)\le G^{-1}\left(\frac{G(m)|\log(G(m))|+\frac{|x|^2}{4}}
 {\Bigl|\log\Bigl(G(m)|\log(G(m))|+\frac{|x|^2}{4}\Bigr)\Bigr|}\right)+o(1).
\end{equation}
 Also, for $(x,t)$ close enough to $(0,T)$,
 we have $G(m)\le G(m)|\log(G(m))|+\frac{|x|^2}{4}<1$,
 so that
 $$\begin{aligned}
 \bigg|\log\bigg(G(m)|\log G(m)|+\frac{|x|^2}{4}\bigg)\bigg|
 &\le\min\left\{|\log(G(m))|,\bigl|\log\bigl(\textstyle\frac{|x|^2}{4}\bigr)\bigr|\right\}\\
 &=(1+o(1))\min\left\{|\log(G(m))|,\bigl|\log|x|^2\bigr|\right\}
 \end{aligned}$$
 as $(x,t)\to (0,T)$, hence
 $$\frac{G(m)|\log(G(m))|+\frac{|x|^2}{4}}
 	{\Bigl|\log\Bigl(G(m)|\log(G(m))|+\frac{|x|^2}{4}\Bigr)\Bigr|}
 	\ge (1+o(1))\left(G(m)+\frac{|x|^2}{4\bigl|\log|x|^2\bigr|}\right).$$
  Since $G^{-1}$ is decreasing near $0^+$, 
  we thus deduce from \eqref{nuitt} 
    that
 \begin{equation}\label{conclprop31}
  u(x,t)\le G^{-1}\left((1+o(1))\left(G(m)+\frac{|x|^2}{4\bigl|\log|x|^2\bigr|}\right)\right)+o(1),
 \quad (x,t)\to (0,T).
 \end{equation}
  On the other hand, by Proposition \ref{CVG} for $x=0$, 
  we have  \begin{equation*}
\lim_{t\to T} \big(m(t)-G^{-1}(T-t)\big)=0,
\end{equation*} 
then by  Lemma~\ref{lem1}(ii) we  obtain \begin{equation}\label{J1}
  G(m)=(1+o(1))(T-t).
  \end{equation}
  Combining \eqref{conclprop31}, \eqref{J1} and applying Lemma~\ref{lem1}(ii) once again, we obtain  \eqref{infer1-upper} and so the upper part of~\eqref{infer1}.
 \end{proof}

\section{Proof of Theorem \ref{theo}} \label{LOT5}

Recalling definition \eqref{defL2H1}, we respectively denote  by $(v,w):=\int_{\mathbb{R}^n}vw\rho dy$ and $\| v\|=(v,v)^{\frac{1}{2}}$ the inner product and the norm of the Hilbert space $L^2_\rho$ (restated in terms of $z=-\varphi$, instead of $\varphi$ in \cite[section 6.3]{chso}). 
The following proposition follows from the proof in \cite[section 6]{chso}.
It was actually assumed there that $\varphi$ is bounded but, owing to the fast decay of the weight $\rho$, the proof remains valid without change whenever $\varphi$ is polynomially bounded at space infinity.

\begin{prop}\label{TheoremBreduit}
Let $z$ be a radially symmetric  nonincreasing solution of  
$$\partial_s z+\mathcal{L}z-z = F(z,s)+\tau(s)z+B(y,s)\quad \hbox{in $\ \R^n\times(s_1,\infty)$,}$$ 
 where 		
$$\tau(s)=o\Bigl(\frac{1}{s}\Bigr),\quad s\to\infty,$$

$$		|B|+|\nabla B|\le Ce^{-s},\quad |F(\xi,s)| \le C\xi^2,  \quad
		|\partial_\xi F(\xi,s)|\le C|\xi|,\quad s\ge s_1,$$
		and 
		$$V(s):=\lim_{\xi\to 0}\frac{F(\xi,s)}{\xi^2}=k+K(s),\quad\hbox{with $k>0$ and } \lim_{s\to \infty}K(s)=0.$$
		We also assume that there exists $M_0> 0$ such that 
		\begin{equation*}
	 |z(y,s)| \le M_0(1+|y|),\quad  |\nabla z|\le M_0 \quad\hbox{ and }\quad 
	\lim_{s\to \infty} z(y,s)=0\ \text{uniformly for } y \text{ bounded,} 
	\end{equation*}
	and that
	$$
	\|z(.,s)\| \geq cs^{-1}, \quad  s\ge  s_1.  
	$$
	Then 
	\begin{equation}
	z(y,s)=\frac{ 2n-|y|^2}{8ks}+o\Bigl(\frac{1}{s}\Bigl)\quad \text{as } s\to \infty,
	\end{equation}
	with convergence in $H^1_\rho(\mathbb{R}^n)$ and uniformly for $y$ bounded.
\end{prop}

 We shall derive Theorem \ref{theo} as consequence of this Proposition.
To this end, we verify the required assumptions in several subsections for clarity.

\subsection{Setting the problem}
	Since we want to apply Proposition \ref{TheoremBreduit}, we first need to extend the solution to $\R^n$ by means of a suitable cut-off in order to handle the case $\Omega=B_R$ (actually, the cut-off procedure can be applied even in the case $\Omega=\R^n$, so as to take advantage of the support 
	and regularity properties of the solution).
	\smallskip
	
	Namely, for $\delta$ given by \eqref{uparabbdry} below,
	we introduce a cut-off function $\phi\in C^\infty(\R^n)$, radially symmetric and nonincreasing in $|x|$, such that 
	\be{defcutoff}
	\phi(x)=\begin{cases}
		1,\quad 0\le |x|\le\delta/2\\
		\noalign{\vskip 1mm}
		0,\quad |x|\ge \delta
	\end{cases}
	\ee
	and we then set $\tilde u(x,t)=\phi(x)u(x,t)$. We see that $\tilde u$ is solution of
	$$
	\tilde u_t-\Delta \tilde u= f(\tilde u)+A(x,t),\quad x\in \mathbb{R}^n,\ 0<t<T
	$$
	where 
	$$A(x,t)=\phi f(u)-f(\phi u)-u\Delta\phi -2\nabla u\cdot\nabla \phi.$$ 
	Note that 
	\be{suppA}
	supp(A)\subset\bigl\{(x,t)\in \mathbb{R}^n\times(0,T);\ \delta/2\le|x|\le\delta\bigr\}
	\quad\hbox{ and }\quad 	|A|+|\nabla A|\le C,
	\ee
	since $0$ is the only blow-up point of $u$. 
	Also, by \eqref{fposM}, \eqref{uparabbdry} and parabolic regularity we see that 
\be{higher-regul}
 \nabla\tilde u\in C^{2,1}(\R^n\times(T-\delta,T)).
\ee

Proposition \ref{TheoremBreduit} will be applied to the function $z$ obtained by rescaling $\tilde u$ by the similarity variables $(y,s)$ (cf.~\eqref{dam111}) and ODE normalization around $(0,T)$:
\be{defw}
 \tilde u(x,t)=z(y,s)+\psi(t),\quad \hbox{ with $\psi$ as in \eqref{defODE-sol}.}
\ee
Let $s_0:=-\log T$. In the rest of the proof, we will denote by $s_1$ a (sufficiently large) time $>\max(s_0,1)$,
	which may vary from line to line.
 By direct calculation, similar to \eqref{dam2}, 
 we see that  $z$ is a global solution of
\begin{equation}\label{(5.2)}
z_s+\mathcal{L}z= Z(z,s) +  B(y,s), \quad y\in \mathbb{R}^n,\ s\in (s_0,\infty),
\end{equation}
where 
$$\mathcal{L}z=-\Delta z+\frac{1}{2}y\cdot\nabla z$$
 is the Hermite operator, and
\be{LO1a}
Z(\xi,s):=h(s)\Bigl(e^\xi\frac{L(e^{\psi_1+\xi})}{L(e^{\psi_1})}-1\Bigr),
\ee
\be{LO1}
h(s)=e^{-s+\psi_1(s)}L(e^{\psi_1(s)}), \quad \psi_1(s)=\psi(T-e^{-s}),
\quad 	 B(y,s):=\frac{A(ye^{-s/2},T-e^{-s})}{e^s}.
\ee
  We then rewrite equation \eqref{(5.2)}  as 
		\be{(5.11)} 
		z_s+\mathcal{L}z-z= W(z,s) +  B(y,s)\equiv \tau(s)z+F(z,s)\ + B(y,s),
		\quad
		y\in\R^n,\ s\ge s_1,	
		\ee
		with

		\be{defWT}
		W(\xi,s):=Z(\xi,s) -\xi, \quad \tau(s):=\partial_\xi W(0,s),\quad F(\xi,s)=W(\xi,s)-\tau(s)\xi.
		\ee
		
		We note that, in the case corresponding to $f(u)=e^u$, we have  $ \psi_1(s)=s$ and then $$h(s)\equiv 1\quad \hbox{and}\quad Z(\xi,s)=e^\xi -1.$$
	Combining with \eqref{defWT}, $W(\xi,s)$ in this case was bounded by a quadratic function. The function $W$ here (for $f(s)=e^sL(e^s)$) has a nonzero linear part, 
	but its decay as $s\to\infty$ will turn out to be sufficiently fast
		so as to permit the required properties. 
 Under the assumptions  of Theorem \ref{Thm}, by \eqref{DwysM1}-\eqref{L12} and Proposition \ref{CVG}, there exist $c,M_0>1$ such that 
	\begin{equation}\label{(1.9)}
	 -c|y|\le z(y,s)\le M_0,\quad  |\nabla z|\le M_0 \quad\hbox{ and }\quad 
	\lim_{s\to \infty} z(y,s)=0\ \text{uniformly for } y \text{ bounded,} 
	\end{equation}
	 and, moreover, by \eqref{suppA},
	\be{suppB}
	 supp(B)\subset\bigl\{(y,s)\in\mathbb{R}^n\times (s_0,\infty): \textstyle\frac{\delta}{2} e^{s/2}\le |y|\le \delta e^{s/2} \bigr\}
		\quad\hbox{ and }\quad 	|B|+|\nabla B|\le Ce^{-s}.
	\ee
		 We also record the following easy bounds for $\psi$:
			\be{compfX}
			\frac{c_1}{T-t}\le f(\psi(t)) \le \frac{c_2}{T-t},\quad t\to T,
			\ee
			for  some constants $c_1, c_2>0$.
			Estimate \eqref{compfX} follows from the fact that
			$T-t=G(\psi(t))\sim\frac{1}{f(\psi(t))}$ as $t\to T$
			(cf.~\eqref{ch} in Lemma~\ref{lem1}).
		
		\smallskip

\subsection{Properties of the equation  \eqref{(5.2)} of $z$}

		\begin{lem}\label{A1709241} We have
			 
		\be{W0s}
		\tau(s)=o\Bigl(\frac{1}{s}\Bigr),\quad s\to\infty,
		\ee
		\be{(5.14)} 
		|F(z,s)| \le Cz^2,  \quad
		|\partial_\xi F(z,s)|\le C|z|,\quad y\in\R^n,\ s\ge s_1
		\ee
		and 
		\be{cvFV}
		\lim_{\xi\to 0}\frac{F(\xi,s)}{\xi^2}=\frac{1}{2}h(s)\Bigl(1+3e^{\psi_1}\frac{L'(e^{\psi_1})}{L(e^{\psi_1})}
		+e^{2\psi_1}\frac{L''(e^{\psi_1})}{L(e^{\psi_1})}\Bigr),\quad  s\ge s_1.
		\ee
	\end{lem}
 
The proof makes use of the following property of the function $Q(X)$.
\begin{lem}\label{lm1b} Assume $L$ is $C^2$ for large $s$ with
\be{used}
\frac{sL'(s)}{L(s)}=o(1)\ \hbox{and }\left(\frac{sL'(s)}{L(s)}\right)'=o\left(\frac{1}{s\log s}\right) \quad \hbox{as } s\to \infty,
\ee
and recall $Q(X)=\int_X^\infty \frac{d\tau}{\tau^2L(\tau)}$. Then
\be{JO}
Q(X)\bigl(XL(X)+X^2L'(X)\bigr)=1+o\Bigl(\frac{1}{\log X}\Bigr), \quad \hbox{as } X\to \infty. 
\ee
\end{lem}
This is a consequence of \cite[Lemma $6.1$]{chso}, but we provide a proof in Appendix for convenience.

	\begin{proof}[Proof of Lemma \ref{A1709241}]
			We  compute
		$$\partial_\xi W(\xi,s)=  \partial_\xi Z(\xi,s)-1=h(s)\Big(e^{\xi}\frac{L(e^{\psi_1+\xi})}{L(e^{\psi_1})}
		+e^{\psi_1-2\xi}\frac{L'(e^{\psi_1-\xi})}{L(e^{\psi_1})}\Big)-1.$$
		Recall $h(s)=e^{-s+\psi_1(s)}L(e^{\psi_1(s)})$ with $\psi_1(s):=\psi(T-e^{-s})=G^{-1}(e^{-s})$.	

Moreover, we have, from \eqref{controlesps}, 
 $Q(e^{\psi(s)})=G(\psi_1(s))= e^{-s}$ for  $s>s_0$ large. Using \eqref{JO} and  \eqref{7.12} we then obtain \eqref{W0s}.

		\smallskip

		To check \eqref{(5.14)}, since $L$ is $C^2$ for $X$ large and $W(0,s)=0$, 
		we may 
		use Taylor's formula with integral remainder to write,
		for $s\ge s_1$ and $\xi\le M_0$,
		\be{TaylorRI} 
		F(\xi,s)=W(\xi,s)-\partial_\xi W(0,s)\xi=
		\xi^2\int^{1}_{0} \partial^2_\xi W(t\xi,s)(1-t)dt,
		\ee
		\be{TaylorRI2}
		\partial_\xi F(\xi,s)=\partial_\xi W(\xi,s)-\partial_\xi W(0,s)
		=\xi\int_0^1 \partial^2_\xi W(t\xi,s)dt,
		\ee
		where
		\be{TaylorRI3}
		\begin{aligned}
			\partial^2_\xi W(\xi,s)
			&=h(s)\Big(e^{\xi}\frac{L(e^{\psi_1+\xi})}{L(e^{\psi_1})}
			+3e^{\psi_1+2\xi}\frac{L'(e^{\psi_1+\xi})}{L(e^{\psi_1})}+e^{2\psi_1+3\xi}\frac{L''(e^{\psi_1+\xi})}{L(e^{\psi_1})}\Big).
		\end{aligned}
		\ee
 By \eqref{limhbeta} and \eqref{limLprime}, $\partial^2_\xi W$
		satisfies the bound
		$$\begin{aligned}
		|\partial^2_\xi W(\xi,s)|
		&\le C\Bigl\{\frac{L(e^{\psi_1+\xi})}{L(e^{\psi_1})}+e^{\psi_1+\xi}\frac{|L'(e^{\psi_1-\xi})|}{L(e^{\psi_1})}
		+e^{2(\psi_1+\xi)}\frac{|L''(e^{\psi_1+\xi})|}{L(e^{\psi_1})}\Bigr\} \\
		&\le C\frac{L(e^{\psi_1+\xi})}{L(e^{\psi_1})}\Bigl\{1+e^{\psi_1+\xi}\frac{|L'(e^{\psi_1+\xi})|}{L(e^{\psi_1+\xi})}
		+e^{2(\psi_1+\xi)}\frac{|L''(e^{\psi_1+\xi})|}{L(e^{\psi_1+\xi})}\Bigr\}\le C,\quad -1\le \xi\le M_0,
		\end{aligned}$$
		 This, combined with \eqref{TaylorRI}-\eqref{TaylorRI2}, yields \eqref{(5.14)} at points $(y,s)$ such that $z(y,s)\ge -1$.

	 To check \eqref{(5.14)} at points such that $ z(y,s)< -1$, in view of \eqref{defWT} and \eqref{W0s}, 
		it suffices to verify that 
		\be{boundW}
		\sup_{s>s_1,\ y\in\R^n} \bigl(|Z(z(y,s),s)|+ |\partial_\xi Z(z(y,s),s)|\bigr)<\infty.
		\ee
		 To this end, we first note that \eqref{sem0}-\eqref{sem} guarantees the existence of 	$X_0$ such that $1/2\le \frac{f'(X)}{f(X)}\le 2$ and $ f(X)$ is increasing on $[X_0,\infty)$,
		and we write
		\be{Wxis}
		Z(\xi,s)=h(s)\Bigl(\frac{f(\psi_1+\xi)}{f(\psi_1)}-1\Bigr),\quad
		\partial_\xi Z(\xi,s)=h(s)\Bigl(\frac{f'(\psi_1+\xi)}{f(\psi_1)}\Bigr),
		\ee
		where we denote $\psi_1=\psi_1(s)$ for simplicity.
		Let $\xi\in (-\infty,M_0]$. If $\psi_1+\xi\ge X_0$, then 
		$0\le \frac{f(\psi_1+\xi)}{f(\psi_1)}\le \frac{f(M_0+\psi_1)}{f(\psi_1)}\le C$ and
		$0\le \frac{f'(\psi_1+\xi)}{f(\psi_1)}\le \frac{2f(\psi_1+\xi)}{ f(\psi_1)}\le\frac{2 f(\psi_1+M_0)}{ f(\psi_1)}\le C$, owing to the fact that $L$ has slow variation at infinity.
		Whereas, if $\psi_1+\xi\le X_0$, then 
		$|\frac{f(\psi_1+\xi)}{f(\psi_1)}|\le \frac{C}{f(\psi_1)}\le C$ and
		$|\frac{f'(\psi_1+\xi)}{f(\psi_1)}|\le C\frac{1}{f(\psi_1)}\le C$. 
		In view of \eqref{limhbeta} and \eqref{Wxis}, this proves~\eqref{boundW}.

Finally, \eqref{cvFV} is a consequence of \eqref{TaylorRI}and  \eqref{TaylorRI3}.	
	\end{proof}
	\smallskip

\subsection{Nonexponential decay of $z$ and completion of proof of Theorem \ref{theo}}
The following lemma provides a 
 lower bound on the decay of $z$ in $L^{2}_{\rho}$, 
 which is one of the key assumptions in Proposition \ref{TheoremBreduit}.
 As in \cite{chso,souplet2019simplified}, this lower bound is obtained as consequence of a lower bound on $|u_r|$ (cf.~\eqref{Rappelupper}).

\begin{lem}\label{lem}
	Under the assumptions of Theorem~\ref{theo}, there  exists $c > 0$ 
	such that
	\be{lowerdecayw}
	\|z(.,s)\| \geq cs^{-1}, \quad  s\ge  s_1.  
	\ee
	
\end{lem}

\begin{proof}
	
	Fix $ K > 0$ and $\eps>0$. By \eqref{(1.9)} there exists $t_0(K,\eps) \in (T/2, T)$ such that  
	$$
	u(r,t) \geq \psi(t)-\eps\ge  M, \quad 0\leq r \leq K\sqrt{T-t}, \ \ t_0 < t < T,
	$$
	where $M$ is as in   \eqref{Rappelupper}.  
	Then, from \eqref{Rappelupper},  $u= u(r,t)$ with  $r=| x |,$ satisfies (see also~\cite[formula $(4.22)$]{chso})
	\begin{equation}\label{(5.10)}
	-u_r(r,t) \geq\frac{rf(u)}{2(A +\log f(u)) },\quad  0 \leq r \leq K\sqrt{T-t}, \ \ t_0 < t < T.
	\end{equation}
	 We write $z =z (\eta, s)$ with $\eta =  |y|$ and $\varphi=-z$, hence $\|z(s)\|=\|\varphi(s)\|$.
	 Observing that the RHS of \eqref{(5.10)} is an increasing function of $u$ for large $u$  and using the properties of $L$, we deduce that, for all 
	$s_2 = s_2(K)\ge  \max(s_0,1) $ sufficiently large, there exists  $c_1=c_1(\eps)>0$,
	\be{roi}
	\begin{aligned}
		\varphi_\eta(\eta, s) &:= -z_\eta (\eta, s) = -\sqrt{T-t} \, u_r(\eta \sqrt{T-t},t) \\
		&\ge
		\frac{ \eta(T-t)f(\psi(t)-\eps)}{2\bigl(A + \log f(\psi(t)-\eps)\bigr)}  \,\ge\, 
		\frac{c_1 \eta(T-t)f(\psi(t))}{\bigl(A + \log f(\psi(t))\bigr)}.
	\end{aligned}
	\ee
	 Inequality \eqref{compfX} guarantees that
		$\log(f(\psi(t)))\le C|\log(T-t)|$ as $t\to T$.
		This combined with \eqref{compfX}, \eqref{roi} implies
	\begin{equation*}
	\varphi_\eta(\eta, s)\ge \frac{c\eta}{s}\quad\hbox{for all $0 \leq \eta \leq K$ and $s > s_2$},
	\end{equation*}
	where the constant $c > 0$ is independent of $K$.
	
	Now, choose $K = 2(1 + c^{-1})$ and take any $s \ge s_2(K).$ If $ \varphi(1,s) \geq -1/s$ then it follows that 
	\begin{equation*} \varphi(\eta, s) = \varphi(1,s ) + \int_{1}^{\eta} \varphi_\eta(z,s)dz \geq \frac{-1 + c(\eta -1)}{s} \geq \frac{1}{s}, \ \ K-1 \leq \eta \leq K, \end{equation*}
	and hence $\|\varphi(\cdot,s)\| \geq \bigl(\int_{K-1\leq | y | \leq K} \rho\bigr)^{1/2}s^{-1}$. Otherwise, we have $\varphi(1,s) \leq -1/s$ and, since $\varphi$ is a nondecreasing function of $\eta$, we get $\varphi(\eta, s) \leq -1/s$ for $\eta \in [0,1];$ hence $\|\varphi (\cdot,s)\| \geq \bigl(\int_{| y | \leq 1} \rho\bigr)^{1/2}s^{-1}.$ We conclude that $\|\varphi (\cdot,s)\| \geq cs^{-1}$ for all $s \geq s_2.$
\end{proof}

\begin{proof}[Proof of Theorem \ref{theo}] Using  \eqref{(5.11)},  \eqref{(1.9)}, \eqref{suppB}, \eqref{lowerdecayw} and  Lemma \ref{A1709241}, the conclusion follows from Proposition \ref{TheoremBreduit} with $k=\frac12$ (due to \eqref{cvFV} and \eqref{limhbeta}). 
\end{proof}

\section{Proof of the lower part of Theorem \ref{Thm} and of Corollary~\ref{coro}}  \label{LOT4}

	We shall derive 
	the lower part of Theorem \ref{Thm}, i.e. 
	estimate \eqref{infer1-lower},
	as a consequence of
	  Theorem \ref{theo}.
	To this end we modify the argument in \cite[Section~5]{chso} for the case $f(u)=u^pL(u)$, 
	the main difference here being the additive instead of multiplicative effect
	of the ODE term $G^{-1}(T-t)$.
	 	\smallskip
		
 Let $ \delta>0$ (to be chosen below). 
 We shall use the following 
 property of $(S_\delta(t))_{t\ge 0}$, the Dirichlet heat semigroup on $B_\delta\subset\mathbb{R}^n$ (see \cite[Lemma $5.1$]{chso}).

 \begin{lem}\label{lemSrho} 
 	(i) There exists a constant  $\lambda=\lambda(n,\delta)>0$ such that, for any $K,N>0$,
 	\begin{equation}
 	\Big[S_\delta(t) \bigl(N-K|x|^2\bigr)\Big]_+\ge e^{-\lambda t}\big[N-K(|x|^2 +2nt)\big]_+
 	\ \hbox{ in $Q:=\overline B_{\delta/2}\times[0,\infty)$.}
 	\end{equation}
 	
 	(ii) For all $\phi\in L^\infty(B_\delta)$, we have
 	\begin{equation}
 	 	\bigl|(S_\delta(t)\phi)(x)\bigr|\le C(n)t^{-n/4}\bigg(\int_{B_\delta}|\phi(z)|^2e^{-\frac{|z|^2}{4t}}dz\bigg)^{1/2}e^{\frac{|x|^2}{2t}},\quad x\in B_\delta, \ t>0.
 \end{equation}
\end{lem}

\smallskip

 \begin{proof}[Proof of \eqref{infer1-lower}]
{\bf Step 1.} {\it Comparison argument.}
 Similar to \eqref{fposM2}, there exists $\delta\in(0,1)$ with $\delta<\min(\sqrt{T},G(M))$, and $\delta<R/4$ if $\Omega=B_R$, such that
		\begin{equation}\label{uparabbdry}
		u(x,t)\ge  \tilde A:=G^{-1}\big(G( A)-\delta\big)> A \ \hbox{ in $Q:=\overline B_\delta\times[T-\delta^2,T)$,}
			\end{equation}
		 where  $A$ is as in \eqref{fposM}.
	Fix any $\sigma \in(0,\delta^2)$ (to be chosen later) and set 
		\begin{equation}\label{deft0}
		t_0=T-\sigma.
			\end{equation}
 As in \cite{chso} (see also \cite{friedman1985blow}), we introduce the following comparison function:
	\begin{equation}\label{defUV}
	U=U_\sigma(x,t)= G^{-1}\Bigl\{G(V)+t_0-t\Bigr\},
	\quad\hbox{ with }V(x,t):= A+S(t-t_0)\bigl(u(t_0)- A\bigr).
	\end{equation}	
	Since $u(t_0)\ge  A$ in $\overline B_\delta$, we have $0\le S(t-t_0)(u(t_0)- A)\le \|u(t_0)\|_\infty- A$, so that $ A\le V\le \|u(t_0)\|_\infty$, hence $\tau_0:=G(\|u(t_0)\|_\infty)\le G(V)\le G( A)$. Letting 
	$$ E:=\{\tau>t_0; G(V)+t_0-t>0,\ \hbox{on } \bar{B}_{\delta}\times[t_0,\tau)\};\  \tau_1=\sup E,\ Q_1:=\bar{B}_\delta\times[t_0,\tau_1),$$
	we see that $E$ is nonempty, $\tau_1\ge t_0+\tau_0$, $U$ is well defined and smooth in $Q_1$.
By the same way as in \cite[section $5$ Step $2$]{chso}, using \eqref{fposM}, we have
	$$\partial_tU-\Delta U\le f(U) \quad\hbox{in $\ Q_1$.}$$
Moreover, setting $T_1=\min\{T,\tau_1\}$, on $\partial B_\delta\times[t_0,T_1)$ we have $V= A$, hence $G(V)+t_0-t\ge G( A)-\delta$, so that $U \le G^{-1}\big(G( A)-\delta\big)= \tilde A$. Since also $U(t_0)=u(t_0)$, it follows from the comparison principle that  $u\ge U$ in $B_\delta\times [t_0,T_1)$ hence in particular $\tau_1\ge T$ (since otherwise $u$ would blow up before $t=T$). Consequently
\begin{equation}\label{a1}
	u\ge U\quad  \hbox{in} \quad B_\delta\times[t_0,T).
\end{equation}
	
		\smallskip

		{\bf Step 2.} {\it Propagation of lower estimate on $u(t_0)$ in the region $|x|=O(\sqrt{T-t_0})$.}
		By Theorem~\ref{theo}, we have
		$$	u(y\sqrt{\sigma},T-\sigma)-G^{-1}(\sigma)=-\frac{|y|^2-2n}{4|\log\sigma|}
		+\frac{\mathcal{R}(y,|\log \sigma|)}{|\log\sigma|}$$
		with
		$\lim_{s\to\infty}\|\mathcal{R}(\cdot,s)\|_{H^1_\rho}=0$, which rewrites in original variables as
		$$	u(x,T-\sigma)= G^{-1}(\sigma)+ \frac{2n}{4|\log\sigma|}-\frac{|x|^2}{4\sigma|\log\sigma|}
		+\frac{\mathcal{R}_\sigma(x)}{|\log\sigma|},
		\quad\hbox{where } \mathcal{R}_\sigma(x)=\mathcal{R}(x\sigma^{-1/2},|\log \sigma|).$$
Therefore we have 
$$\begin{aligned}
V(x,t)&= A-S(t-t_0) A+S(t-t_0)u(t_0)\ge S(t-t_0)u(t_0)\\
&\ge S(t-t_0)\Bigl(G^{-1}(\sigma)+\frac{2n}{4|\log\sigma|}-\frac{|x|^2}{4\sigma|\log\sigma|}\Bigl)
+\frac{1}{|\log\sigma|}S(t-t_0)\mathcal{R}_\sigma\equiv V_1+V_2\\
\end{aligned}$$
hence, since $V\ge 0$,
\be{VV1V2}
V\ge (V_1)_+-|V_2|\quad \hbox{ in $B_\delta\times(t_0,T)$.}
\ee
Using Lemma~\ref{lemSrho}(i) and recalling \eqref{deft0}, 
we may estimate $(V_1)_+$ from below for $t\in[t_0,T)$ as:
\be{V1below}
\begin{aligned}
(V_1(t))_+
&\ge e^{-\lambda(t-t_0)} \Bigl[G^{-1}(\sigma)+\frac{2n}{4|\log\sigma|}-\frac{|x|^2+2n(t-t_0)}{4\sigma|\log\sigma|}\Bigl]_+
\chi_{B_{\delta/2}} \\
&{\strut =\strut} e^{-\lambda\sigma} \Bigl[G^{-1}(\sigma)-\frac{|x|^2}{4\sigma|\log\sigma|}\Bigl]_+
\chi_{B_{\delta/2}}. 
\end{aligned}
\ee
To control $V_2$, we observe that, for $t\in[t_0,T)$,
$$\int_{B_\delta}\mathcal{R}^2_\sigma(z)e^{-\frac{|z|^2}{4(t-t_0)}}dz
\le\int_{B_\delta} \mathcal{R}^2(z\sigma^{-1/2},|\log \sigma|)e^{-\frac{|z|^2}{4\sigma}}dz
\le\sigma^{n/2}\|\mathcal{R}(\cdot,|\log \sigma|)\|^2_{L^2_\rho}$$
and then we use Lemma~\ref{lemSrho}(ii) 
to write
$$\begin{aligned}
\bigl|S(t-t_0)\mathcal{R}_\sigma\bigr|(x)
&\le C(n)(t-t_0)^{-n/4}\bigg(\int_{B_\delta}\mathcal{R}^2_\sigma(z)e^{-\frac{|z|^2}{4(t-t_0)}}dz\bigg)^{1/2}e^{\frac{|x|^2}{2(t-t_0)}}\\
&\le C(n)(t-t_0)^{-n/4}\sigma^{n/4}e^{\frac{|x|^2}{2(t-t_0)}}\|\mathcal{R}(\cdot,|\log \sigma|)\|_{L^2_\rho},
\quad  x\in B_\delta,\ t\in[t_0,T).
\end{aligned}$$
Denoting $D_\sigma=\overline B_{\sqrt\sigma}\times [T-\textstyle\frac{\sigma}{2},T)$ and using \eqref{deft0},
we obtain
\be{controlV2}
\sup_{D_\sigma}|V_2|
\le C(n)\frac{\|\mathcal{R}(\cdot,|\log \sigma|)\|_{L^2_\rho}}{|\log\sigma|}
=\frac{\eps(\sigma)}{|\log\sigma|}.
\ee
Here and in the rest of the proof $\eps$ denotes a generic function such that $\lim_{\sigma\to 0}\eps(\sigma)=0$.
 Assuming $\sigma\le \delta^2/4$, it follows from \eqref{VV1V2}-\eqref{controlV2} that 
\be{controlV}
V\ge e^{-\lambda\sigma} \Bigl(G^{-1}(\sigma)-\frac{1+\eps(\sigma)}{4|\log\sigma|}\Bigl)\quad\hbox{in $D_\sigma$}.
\ee
Moreover, we may choose $\sigma_0\in(0,\delta^2/4)$ such that the RHS of \eqref{controlV} is $> A$ 
for all $\sigma\in(0,\sigma_0]$.
Therefore
$$
0<G(V)+t_0-t\le  \mu(\sigma,t):=G\biggl[ e^{-\lambda\sigma} \Bigl(G^{-1}(\sigma)-\frac{1+\eps(\sigma)}{4|\log\sigma|}\Bigl)\biggr]+T-t-\sigma
\quad\hbox{in $D_\sigma$},
$$
where the RHS is $<G( A)$ hence, by \eqref{defUV}, \eqref{a1},
\be{fact}
u\ge G^{-1}(\mu(\sigma,t))\quad\hbox{in $D_\sigma$}.
\ee

{\bf Step 3.} {\it Asymptotics of $\mu(\sigma,t)$ and conclusion.}
	 Fix any  $x$ such that $|x|^2\le\sigma_0$. Choosing  
		\be{choicesigma}
		\sigma=|x|^2,
		\ee
		it follows from \eqref{fact} that
		\begin{equation}\label{fact2}
		u(x, t)\ge G^{-1}(\mu(\sigma,t))\quad\hbox{ for $T-\frac{\sigma}{2}\le t<T$.}
	\end{equation}
	 To control $\mu$, we shall use the following lemma, which complements Lemma \ref{lem1}(ii).
	 \begin{lem}\label{lemconfG}
	There exist $X_1,C_1>0$ such that, for all $\eps\in (0,1/2)$ and $X\ge X_1$,
	\be{controfeps}
	f(X-\eps)\ge (1-\eps)^2f(X)
	\ee
	and
	\be{controGeps}
	G(X-\eps)\le (1+C_1\eps)G(X).
	\ee
	\end{lem}
	\begin{proof}[Proof of Lemma \ref{lemconfG}]
 Using Lemma~\ref{ell8}	, there exists $X_1>0$ such that,  for all $\eps \in (0,1/2)$ and $X\ge X_1$,
	 $L(e^{X-\eps})/L(e^X)\ge 1-\eps$, so that
	$$ f(X-\eps)=e^{-\eps} e^X L(e^{X-\eps})\ge f(X)e^{-\eps}(1-\eps)\ge (1-\eps)^2f(X),$$
	hence \eqref{controfeps}. From the later, we deduce
	$$
	G(X-\eps)=\int_{X-\eps}^\infty \frac{ds}{f(s)}=\int_{X}^\infty \frac{dt}{f(t-\eps)}\le (1-\eps)^{-2} G(X),
	$$
	and \eqref{controGeps} follows.
	\end{proof}
	
	 Recall that $G^{-1}(\sigma)\sim |\log\sigma|$ as $\sigma\to 0^+$ (cf.~Lemma \ref{lm1}(ii)) 
	and that $G, G^{-1}$ are decreasing near $\infty$ and $0^+$ respectively.
	Now, for any $\sigma>0$ sufficient small, we first use \eqref{controGeps} to get, for some $C>0$,
	 	\begin{align*}
		  G\biggl[e^{-\lambda\sigma} \Bigl(G^{-1}(\sigma)-\frac{1+\eps(\sigma)}{4|\log\sigma|}\Bigl)\biggr]\le&  
			 G\biggl[G^{-1}(\sigma)-\frac{1+\eps(\sigma)}{4|\log\sigma|}-\lambda\sigma\Big( G^{-1}(\sigma)-\frac{1+\eps(\sigma)}{4|\log\sigma|}\Big)\biggr]\\
	 \le& \Bigg(1+C\sigma\Big( G^{-1}(\sigma)-\frac{1+\eps(\sigma)}{4|\log\sigma|}\Big)\Bigg)G\biggl[G^{-1}(\sigma)-\frac{1+\eps(\sigma)}{4|\log\sigma|}\biggr].		
	 \end{align*}
Using the mean value theorem, there exists $\theta(\sigma)\in(0,1)$ such that 
$$G\biggl[G^{-1}(\sigma)-\frac{1+\eps(\sigma)}{4|\log\sigma|}\biggr]=\sigma -G'\biggl[G^{-1}(\sigma)-\theta(\sigma)\frac{1+\eps(\sigma)}{4|\log\sigma|}\biggr] \frac{1+\eps(\sigma)}{4|\log\sigma|}.$$
 By \eqref{ch} in Lemma \ref{lm1} and \eqref{controfeps}, we have
\begin{align*}
-G'\Bigl[G^{-1}(\sigma)-\theta(\sigma)\frac{1+\eps(\sigma)}{4|\log\sigma|}\Bigr] 
&=\frac{1}{f\bigl[G^{-1}(\sigma)-\theta(\sigma)\frac{1+\eps(\sigma)}{4|\log\sigma|}\bigr]}
=(1+\eps(\sigma))G\Bigl[G^{-1}(\sigma)-\theta(\sigma)\frac{1+\eps(\sigma)}{4|\log\sigma|}\Bigr]\\
&=(1+\eps(\sigma))\Bigl[1+\frac{C_1}{|\log\sigma|}\Bigr] \sigma
\le (1+\eps(\sigma))\sigma.
\end{align*}	
Consequently,
\begin{align*}
\mu(\sigma,t)
&\le \bigg(1+C\sigma\Big( G^{-1}(\sigma)-\frac{1+\eps(\sigma)}{4|\log\sigma|}\Big)\bigg)
\Bigl( \sigma+\frac{(1+\eps(\sigma))\sigma}{4|\log \sigma|}\Bigr)+T-t-\sigma\\
&\le \Big(1+C\sigma G^{-1}(\sigma)\Bigr)\Bigl(1+\frac{1+\eps(\sigma)}{4|\log \sigma|}\Bigr) \sigma+T-t-\sigma\\
&= T-t+\sigma \frac{1+\eps(\sigma)}{4|\log \sigma|}+C\sigma^2G^{-1}(\sigma)\le (1+\eps(\sigma))\Big(T-t+\frac{\sigma}{4|\log \sigma|}\Big).
\end{align*}		
By \eqref{choicesigma}, \eqref{fact2}, using Lemma \ref{lem1}(ii),
we thus deduce that, as $(x,t)\to (0,T)$,
		$$u(x,t)\ge G^{-1}\bigg(T-t+\frac{|x|^2}{4|\log|x|^2|}\bigg)+o(1),
		\quad\hbox{ uniformly for $ |x|^2\ge 2(T-t)$.}$$
		
	On the other hand, using \eqref{1.9} in Proposition \ref{CVG} 	and \eqref{equivGG} again, 
	 we have, as $(x,t)\to (0,T)$,
		$$u(x,t)= G^{-1}(T-t)+o(1)=G^{-1}\bigg(T-t+\frac{|x|^2}{4|\log|x|^2|}\bigg)+o(1),
		\ \hbox{ uniformly for $ |x|^2\le 2(T-t)$.}$$
		 Estimate \eqref{infer1-lower}, and so the lower part of \eqref{infer1}, follows.
 \end{proof}

\begin{proof}[Proof of Corollary \ref{coro}]
	(i) Letting $t\to T$ in the RHS of \eqref{infer1} and  using the continuity of $G^{-1}$, we get
$$
u(x,T):=\lim_{t\to T} u(x,t)= G^{-1}\Bigl(\frac{|x|^2}{4\log|x|^2|}\Bigr) +o(1)\quad \text{as } x\to 0.
$$

(ii) Let $K>0$.
In view of \eqref{infer1} and \eqref{equivGG}, setting $\xi=x/\sqrt{(T-t)|\log(T-t)|}$, 
 it suffices to show
that
\be{limdeltaK}
\lim_{t\to T} \Bigl(\sup_{0<|\xi|\le K} |\delta(x,t)|\Bigr)=0,
\quad\hbox{where }\delta(x,t):=\displaystyle\frac{T-t+\frac{|x|^2}{4|\log|x|^2|}}{(T-t)(1+\frac{|\xi|^2}{4})}-1.
\ee
We first control $\delta(x,t)$ as follows
$$|\delta(x,t)|=\frac{{ |\xi|^2}}{4}\frac{\bigl|\frac{|\log(T-t)|}{|\log|x|^2|}-1\bigr|}{1+\frac{|\xi|^2}{4}}
\le \biggl|\frac{|\log(T-t)|}{|\log|x|^2|}-1\biggr|{ |\xi|^2}.$$
If $0<{ |\xi|^2}|\log(T-t)|\le 1$ then 
$|\log(T-t)|\le |\log|x|^2|$ for $t$ close to $T$, hence
$|\delta(x,t)|\le |\xi|^2\le \frac{1}{|\log(T-t)|}$.
Whereas if ${ |\xi|^2}|\log(T-t)|\ge 1$ and $|\xi|\le K$ then, for all $\eps>0$, we have
$(1+\eps)^{-1}|\log(T-t)|\le |\log|x|^2|\le |\log(T-t)|$ for $t$ close to $T$,
hence $|\delta(x,t)|\le K^2\eps$. Property \eqref{limdeltaK} follows, 
 hence Corollary \ref{coro}(ii). 
\end{proof}

\begin{rmq}  \label{remtypeIb}
Let us justify the claim made in Remark~\ref{remtypeI}, that, for $n\le 2$, Theorems \ref{Thm}-\ref{theo} remain valid 
if we replace the time increasing assumption $u_t\ge0$ by 
the type I blow-up condition \eqref{dam1a}.
Indeed the assumption $u_t\ge0$ was used only at the following places:

\begin{itemize}
\item[-] to ensure \eqref{dam1a} (cf.~Lemma~\ref{lemtypeI}), which is used later on;

\item[-]  to rule out nonzero limits $\phi\in\mathcal{S}$ when deducing 
Proposition \ref{CVG} from Proposition \ref{cvg} (see~Lemma~\ref{cvg0}).
But for $n\le 2$, we have $\mathcal{S}=\{0\}$  by \cite[Lemma 4.3(a)]{BE};

\item[-] to guarantee the largeness of $u$ 
in the proof of the lower estimate of Theorem \ref{Thm} (cf.~\eqref{uparabbdry}).
But, as a consequence of Proposition \ref{CVG}, the ``no-needle'' property $\lim_{(x,t)\to(0,T)}u(x,t)=\infty$
can be shown by a similar argument as in \cite[Proposition~2.1]{chabi1}, which does not require $u_t\ge 0$.
\end{itemize}
\end{rmq}

\section{Appendix: Proof of Lemma~\ref{ell8}, \ref{cvg0} and \ref{lm1b}}
 
We first recall the proof of Lemma~\ref{ell8} (cf. \cite[Lemma 7.2]{chso}), which is
used in the proofs of Lemma~\ref{lm1} and \ref{lemconfG}.
\begin{proof}[Proof of Lemma~\ref{ell8}]
By assumption \eqref{sem11}, there exists $s_0>1$ such that 
\be{ell8.1}
\frac{|L'(s)|}{L(s)}\le \frac{1}{s\log^\alpha{\hskip -2.5pt}s}\quad\hbox{ for all $s\ge s_0$.}
\ee
Also, since $s\exp(-\frac18\log^\alpha{\hskip -2.5pt}s)\to \infty$ 
as $s\to\infty$, there exists $s_1>s_0$ such that $s\exp(-\frac18\log^\alpha{\hskip -2.5pt}s)>s_0$ for all $s\ge s_1$. 
It suffices to prove that, for each $\eps\in(0,\frac12)$, there holds
\be{ell8.20}
\Bigl|\frac{L(\lambda s)}{L(s)}-1\Bigl|\ \le \eps+4\frac{|\log \lambda|}{\log^\alpha{\hskip -2.5pt}s}
\quad\hbox{ for all $\lambda\in I_s$.}
\ee

Fix $s\ge s_1$ and assume for contradiction that there exists $\eps\in(0,\frac12)$ such that \eqref{ell8.20} fails.
Since the inequality in \eqref{ell8.20} is true for $\lambda$ close to $1$ by continuity of $L$, 
there must exist $\lambda_0\in I_s\setminus\{1\}$ such that
\be{ell8.2}
\Bigl|\frac{L(\lambda s)}{L(s)}-1\Bigl|\ \le \eps+4\frac{|\log \lambda|}{\log^\alpha{\hskip -2.5pt}s}\ \hbox{ for all $\lambda \in (a,b)$\quad  and }\quad
	\Bigl|\frac{L(\lambda_0 s)}{L(s)}-1\Bigl|\ =\eps+4\frac{|\log \lambda_0|}{\log^\alpha{\hskip -2.5pt}s},
	\ee
	where either $a=\lambda_0<b=1$ or $a=1<b=\lambda_0$.
		  On the other hand, using the elementary inequality $|(1+h)^{1-\alpha}-1|\le 2(1-\alpha)|h|$ for $\alpha\in(0,1)$ and $h\in(-1,1)$, we compute
		 	$$\begin{aligned}
			\int_{as}^{bs}\frac{dz}{z\log^\alpha z}
			&=(1-\alpha)^{-1}\bigl|\log^{1-\alpha}(bs)-\log^{1-\alpha}(as)\bigr|
			 =(1-\alpha)^{-1}\bigl|\bigl(\log s\pm|\log\lambda_0|\bigr)^{1-\alpha}-\log^{1-\alpha}s\bigr| \\
		 	&=(1-\alpha)^{-1}\log^{1-\alpha}s\Bigl|\Bigl(1\pm\frac{|\log\lambda_0|}{\log s}\Bigr)^{1-\alpha}-1\Bigr|
				 \le 2\frac{|\log\lambda_0|}{\log^\alpha{\hskip -2.5pt}s}.
				 				 	 	\end{aligned} $$
It then follows from \eqref{ell8.1}, \eqref{ell8.2} 
and $\lambda_0\in I_s\setminus\{1\}$ that
	$$\begin{aligned} 
	4\frac{|\log \lambda_0|}{\log^\alpha{\hskip -2.5pt}s}
	 &<\Bigl|\frac{L(\lambda_0 s)}{L(s)}-1\Bigl|\ 
	 =\frac{|L(\lambda_0 s)-L(s)|}{L(s)}
	 \le  \int_{as}^{bs} \frac{|L'(z)|}{L(s)}dz
	 	 \le  \int_{as}^{bs} \frac{L(z)}{L(s)}\frac{dz}{z\log^\alpha z}\\
		 &\le \Bigl(1+\eps+4\frac{|\log \lambda_0|}{\log^\alpha{\hskip -2.5pt}s}\Bigl)\int_{as}^{bs}
		 \frac{dz}{z\log^\alpha z}
		 		 < 2\Bigl(\frac32+4\frac{|\log \lambda_0|}{\log^\alpha{\hskip -2.5pt}s}\Bigl)\frac{|\log\lambda_0|}{\log^\alpha{\hskip -2.5pt}s}
		 \le 4\frac{|\log\lambda_0|}{\log^\alpha{\hskip -2.5pt}s}:	\end{aligned} $$
		a contradiction.
		\end{proof}

 We next give a proof of Lemma~\ref{cvg0}, that we used in the proof of Proposition \ref{CVG}.
We note that our proof is significantly simpler than that of~\cite[Lemma 4.2(a)]{BE},
owing to the introduction of the auxiliary function $Z$ below.

\begin{proof}[Proof of Lemma \ref{cvg0}] 
 Assume for contradiction that $\phi \not\equiv 0$, hence $\phi(0)>0$ by local uniqueness for problem \eqref{ODEphi}.
 By direct calculation, we have
\be{eqg}
g''+c(r)g'+F(\phi)g=0,\quad g(0)>0,\ g'(0)=0.
\ee
 Since $g\ge 0$ by assumption and $g\not\equiv 0$, 
local uniqueness for \eqref{eqg} yields that actually $g>0$ on $[0,\infty)$.
Set 
$$Z:=e^{2\phi}(ge^{-\phi})'=e^\phi(g'-\phi'g).$$
Using \eqref{ODEphi} and \eqref{eqg}, we compute, for $r>0$:
$$\begin{aligned}
e^{-\phi}Z'
=\phi'(g'-\phi'g)+(g''-\phi''g-\phi'g')
=-{\phi'}^2g-cg'-F(\phi)g+g(c\phi'+F(\phi))
=-{\phi'}^2g-c(g'-\phi'g).
\end{aligned}$$
Letting $\mu(r)=r^{n-1}e^{-r^2/4}$, we deduce
\be{eqg2}
(\mu Z)'=\mu(Z'+cZ)=-\mu{\phi'}^2g e^\phi\le 0,\quad r>0.
\ee
Note that $\phi''(0)=-n^{-1}F(\phi(0))<0$, hence $\phi'<0$ for $r>0$ small.
Since $Z(0)=0$, integrating \eqref{eqg2} and using also $g>0$ and $e^\phi$ nonincreasing, 
we obtain, for some constant $C>0$,
$$Z(r)=-\mu^{-1}(r)\int_0^r \bigl[\mu{\phi'}^2g e^\phi\bigr](s)ds\le -Cr^{1-n}e^{r^2/4}e^{\phi(r)},\quad r\ge 1,$$
hence
$$(ge^{-\phi})'(r)=[e^{-2\phi}Z](r)\le -Cr^{1-n}e^{r^2/4}e^{-\phi(r)}\le -Cr^{1-n}e^{-\phi(0)}e^{r^2/4},\quad r\ge 1.$$
But this implies $(ge^{-\phi})(r)\to -\infty$ as $r\to\infty$, contradicting $g\ge 0$.
\end{proof}

We end this section by proving Lemma~\ref{lm1b} used in the proof of Lemma~\ref{A1709241}.
\begin{proof}[Proof of Lemma~\ref{lm1b}]
Integrating by parts, for $X$ large enough, we have
$$
Q(X)=\int_{X}^{\infty}\frac{ds}{s^2L(s)}=\frac{X^{-1}}{L(X)}-\int_{X}^{\infty}\frac{s^{-1}L'(s)}{L^2(s)}ds.
$$
By the first part of \eqref{used}, we deduce $Q(X)-\frac{1}{XL(X)}=o(Q(X))$, hence 
\be{used1}
Q(X)\sim \frac{1}{XL(X)}\qquad\hbox{as } X\to \infty.
\ee
Integrating by parts a second time yields
\begin{align*}
Q(X)&=\frac{X^{-1}}{L(X)}+\int_{X}^{\infty}Q'(s)\frac{sL'(s)}{L(s)}ds=\frac{X^{-1}}{L(X)}-\frac{XL'(X)}{L(X)}Q(X)
-\int_{X}^{\infty}Q(s)\Bigl(\frac{sL'(s)}{L(s)}\Bigl)'ds,
\end{align*}
hence
$$\bigl(XL(X)+X^2L'(X)\bigr)Q(X)-1=\mathcal{J}(X):=-XL(X)\int_{X}^{\infty}Q(s)\Bigl(\frac{sL'(s)}{L(s)}\Bigl)'ds.
$$
By the second part of \eqref{used} and \eqref{used1}, we have 
$$\mathcal{J}(X)=o\bigl(X L(X)\mathcal{\tilde J}(X)\bigr)\ \ \hbox{as $X\to\infty$,\quad where }
\mathcal{\tilde J}(X):=\int_{X}^{\infty}\frac{s^{-2}}{L(s)\log s}ds.$$
 To estimate $\mathcal{\tilde J}(X)$, for all $Y>X$, integrating by parts, we write
$$\int_X^Y \frac{s^{-2}}{L(s)\log s}ds
=\Bigl[\frac{-s^{-1}}{L(s)\log s}\Bigr]_X^Y-\int_X^Y s^{-2}\frac{s(L(s)\log s)'}{(L(s)\log s)^2}ds.$$
Next using $\lim_{s\to\infty} \frac{s(L(s)\log s)'}{L(s)\log s}
=\lim_{s\to\infty} \bigl(\frac{sL'(s)}{L(s)}+\frac{1}{\log s}\bigr)=0$, we get
$$\left|\int_X^Y \frac{s^{-2}}{L(s)\log s}ds+\Bigl[\frac{s^{-1}}{L(s)\log s}\Bigr]_X^Y\right|
\le \eps(X)\int_X^Y \frac{s^{-2}}{L(s)\log s}ds.$$
Since the first part of \eqref{used} implies 
$
\frac{1}{L(s)}=o(s^\eta)\ \ \hbox{as $s\to\infty$,\ for any $\eta>0$,}
$
we may let $Y\to\infty$ to deduce
$\mathcal{\tilde J}(X)\sim\frac{X^{-1}}{L(X)\log X}$,
hence $\mathcal{J}(X)=o(1/\log X)$, as $X\to\infty$. This proves the lemma.
\end{proof}

\noindent{\bf Acknowlegement.} The author thanks Prof.~Philippe Souplet for useful suggestions during the preparation of this work.

\smallskip

\noindent{\bf Statements and Declarations.} The author states that there is no conflict of interest. This manuscript
has no associated data.

\end{document}